\documentclass[11pt]{amsart}

 \pdfoutput=1
\usepackage{comment}

\usepackage{hyperref} 
\usepackage[dvips]{graphicx}
\usepackage{color}
\usepackage{multicol}
\usepackage[T1]{fontenc}
\usepackage{mathptmx}

\newcommand{\bbar}{\begin{pmatrix}}
\newcommand{\ebar}{\end{pmatrix}}

\newcommand{\uu}{{\Lambda G_\sigma}}
\newcommand{\bdm}{\begin{displaymath}}
\newcommand{\edm}{\end{displaymath}}
\newcommand{\beq}{\begin{equation}}
\newcommand{\beqa}{\begin{eqnarray}}
\newcommand{\beqas}{\begin{eqnarray*}}
\newcommand{\eeq}{\end{equation}}
\newcommand{\eeqa}{\end{eqnarray}}
\newcommand{\eeqas}{\end{eqnarray*}}
\newcommand{\dd}{\textup{d}}
\newcommand{\Ad}{\textup{Ad}}
\newcommand{\Ord}{\textup{Ord}}
\newcommand{\E}{{\mathbb E}}

\newcommand{\C}{{\mathbb C}}

\newcommand{\real}{{\mathbb R}}
\newcommand{\SSS}{{\mathbb S}}

\newcommand{\sym}{\mathcal{S}}

\newcommand{\ustar}{\Lambda_P^+ G^\C_\sigma}
\newcommand{\uc}{\Lambda G^\C_\sigma}

   \newtheorem{theorem}{Theorem}[section]
   \newtheorem{proposition}[theorem]{Proposition}
   
   \newtheorem{lemma}[theorem]{Lemma}
   \newtheorem{defn}[theorem]{Definition}

 \theoremstyle{remark}
   \newtheorem{example}[theorem]{Example}
   \newtheorem{remark}[theorem]{Remark}
   
\numberwithin{equation}{section}

\begin{document}

\title[Deformations of CMC surfaces]{Deformations of constant mean curvature surfaces preserving symmetries and the Hopf differential}

\author{David Brander}
\address{
Institut for Matematik og Computer Science,
Matematiktorvet, bygning 303B,
Technical University of Denmark}
\email{D.Brander@mat.dtu.dk}

\author{Josef F. Dorfmeister}
\address{TU M\"unchen\\ Zentrum Mathematik (M8), Boltzmannstr. 3\\
  85748, Garching\\ Germany}
\email{dorfm@ma.tum.de}

\thanks{Research partially supported by FNU grant \emph{Symmetry Techniques in Differential Geometry}}

\begin{abstract}
We define certain deformations between minimal and non-minimal constant mean curvature (CMC) surfaces in Euclidean space $\E^3$ which preserve the Hopf differential. We prove that, given a CMC $H$ surface $f$, either minimal or not, and a fixed basepoint $z_0$ on this surface, there is a naturally defined family $f_h$, for all $h \in \real$, of CMC $h$ surfaces that are tangent to $f$ at $z_0$, and which have the same Hopf differential.  Given the classical Weierstrass data for a minimal surface, we give an explicit formula for the generalized Weierstrass data for the non-minimal surfaces $f_h$, and vice versa. As an application, we use this to give a well-defined dressing action
on the class of minimal surfaces.  In addition, we show that symmetries of certain types associated with the basepoint are preserved under the deformation, and this gives a canonical choice of basepoint for surfaces with symmetries. We use this to define new examples of non-minimal CMC surfaces naturally associated to known minimal surfaces with symmetries.
\end{abstract}



\maketitle

\section{Introduction}
Let $\Sigma \subset \C$ be a simply connected domain.  The classical Weierstrass representation for minimal surfaces states that given a pair $(\dd \omega , \, \nu)$,
where $\dd \omega = \mu (z) \dd z$ is a holomorphic $1$-form and $\nu$ a meromorphic function on
$\Sigma$,   and appropriate orders of vanishing, then
the formula
\bdm
f = 2 \Re \int_{z_0}^z f_z \, \dd z, \quad \quad
f_z \, \dd z=  \left((1-\nu^2)e_1 -  i  (1+\nu^2) \, e_2 -  2  \nu \, e_3\right)\, \dd \omega,
\edm
gives a minimal surface $f: \Sigma \to \real^3$. Conversely, all minimal immersions of $\Sigma$ can be obtained this way. This representation is one of the major tools in the study of minimal surface theory.  \\

For non-minimal constant mean curvature (CMC) surfaces, an infinite dimensional analogue to the Weierstrass representation was given in the 1990's by Dorfmeister, Pedit and Wu in \cite{DorPW}.  If we take the above $1$-form $f_z \, \dd z$ as the Weierstrass data for a minimal surface, then the analogous coordinate-independent data for a non-minimal surface is the \emph{normalized potential} $\hat \eta$, which
can be expressed in local coordinates as
\beq   \label{standardpotential}
\hat \eta = \bbar  0 & -\frac{H}{2} a(z) \\ \frac{Q(z)}{a(z)} & 0 \ebar \lambda^{-1} \dd z,
\eeq
where $Q$ is a holomorphic function, and $a$ is meromorphic (with appropriate orders of vanishing), and $\lambda$ is an  $\SSS^1$-parameter. The holomorphic bilinear form $Q \, \dd z^2$ is called the Hopf differential and is well-defined independent of coordinates.  The surface is obtained by integrating $\hat \eta$, performing a loop group decomposition, and applying the Sym-Bobenko formula, a simple formula involving the factor $1/H$.  The loop group decomposition is non-trivial to 
write down explicitly in practice, which means that it is more difficult to use this representation to construct or study CMC surfaces when compared with minimal surfaces. \\

The $1$-forms $f_z \dd z$ and $\hat \eta$  are unique given a choice of  basepoint $z_0$ on $\Sigma$.   Although one cannot substitute $H=0$ directly into the Sym-Bobenko formula,
it is a plausible guess that taking the limit
as $H$ tends to zero in the potential (\ref{standardpotential}) might lead to a minimal
surface.   The main result of this article is more useful than that, because it includes the converse:

\begin{theorem}   \label{mainthm}
Let $\Sigma \subset \C$ be a contractible domain.
\begin{enumerate}
\item  \label{mainthmitem1}
Let $f: \Sigma \to \E^3$ be a conformally immersed minimal surface, 
and a basepoint $z_0 \in \Sigma$ fixed.
  Then there is a canonical family of 
conformally immersed CMC $h$ surfaces $f_h: \Sigma \to \E^3$, 
all with the same Hopf differential as $f$, such that $f=f_0$ and
all of the maps $f_h$, together with their tangent planes, agree 
at the point $z_0$.  The family depends real analytically on $h$.
If $(\mu \dd z , \nu)$  are the
classical Weierstrass data for $f$, with coordinates chosen so that $\mu(z_0)=1$ and $\nu(z_0)=0$, then
 the normalized potential  for $f_h$  is 
\beq   \label{munupotential}
\hat \eta_h = \bbar 0 & -h \mu  \\ -\nu_z & 0 \ebar \lambda^{-1} \dd z.
\eeq

\item  \label{mainthmitem2}
Conversely, let $H$ be any non-zero real number and 
$f_H: \Sigma \to \E^3$ be a CMC $H$ immersion.  For a given basepoint $z_0$
 let 
\beq  \label{Hpotential}
\hat \eta_H = \bbar 0 & -\frac{H}{2}a \\ \frac{Q}{a} & 0 \ebar \lambda^{-1} \dd z,
\eeq
be the associated  normalized potential.
For any meromorphic function $g$, let $\Ord(g(z))$ denote the order of vanishing of $g$ at 
the point $z$.
Let $\Sigma^*$ be the open dense subset of $\Sigma$ on which the
following conditions are satisfied:
\begin{enumerate}
\item \label{conda} the function $a$ is holomorphic;
\item \label{condb}  at any zero of $a$ we have $\Ord(Q) = (\Ord(a)-2)/2$.
\end{enumerate}
Then the surface $f_H\big|_{\Sigma^*}: \Sigma^*\to \real$ is part of a family $f_h: \Sigma^* \to \E^3$, of
CMC immersions, for all $h \in \real$, and the normalized potential for $f_h$ is given by substituting
$h$ for $H$ in (\ref{Hpotential}). 
If  coordinates are chosen such that 
$a(z_0)$ is a real number, then 
the minimal surface $f_0: \Sigma^* \to \E^3$
has the classical Weierstrass data 
\bdm
\mu (z) \dd z = \frac{a(z)}{2}\dd z, \quad \quad \nu = -\int_{z_0}^z \frac{ Q(\tau)}{ a(\tau)} \dd \tau.
\edm
\end{enumerate}
\end{theorem}

The proof of this theorem is given in Section \ref{thmproofsection}. In item (2), the map $f_0$ is defined on 
the whole of $\Sigma$, but may have branch points at poles or zeros of $a$.   Note also that, if the basepoint $z_0$ is changed to a different basepoint $\tilde z_0$, then the resulting family $\tilde f_h$ is not the same family as $f_h$.
We also remark that if, in item (1), the data are given such that $\nu(z_0)\neq 0$, then there is an alternative, more general, formula for $\hat \eta_h$ given below at (\ref{generaleta}). \\

\begin{figure}[ht]
\centering
$
\begin{array}{cc}
\includegraphics[height=55mm]{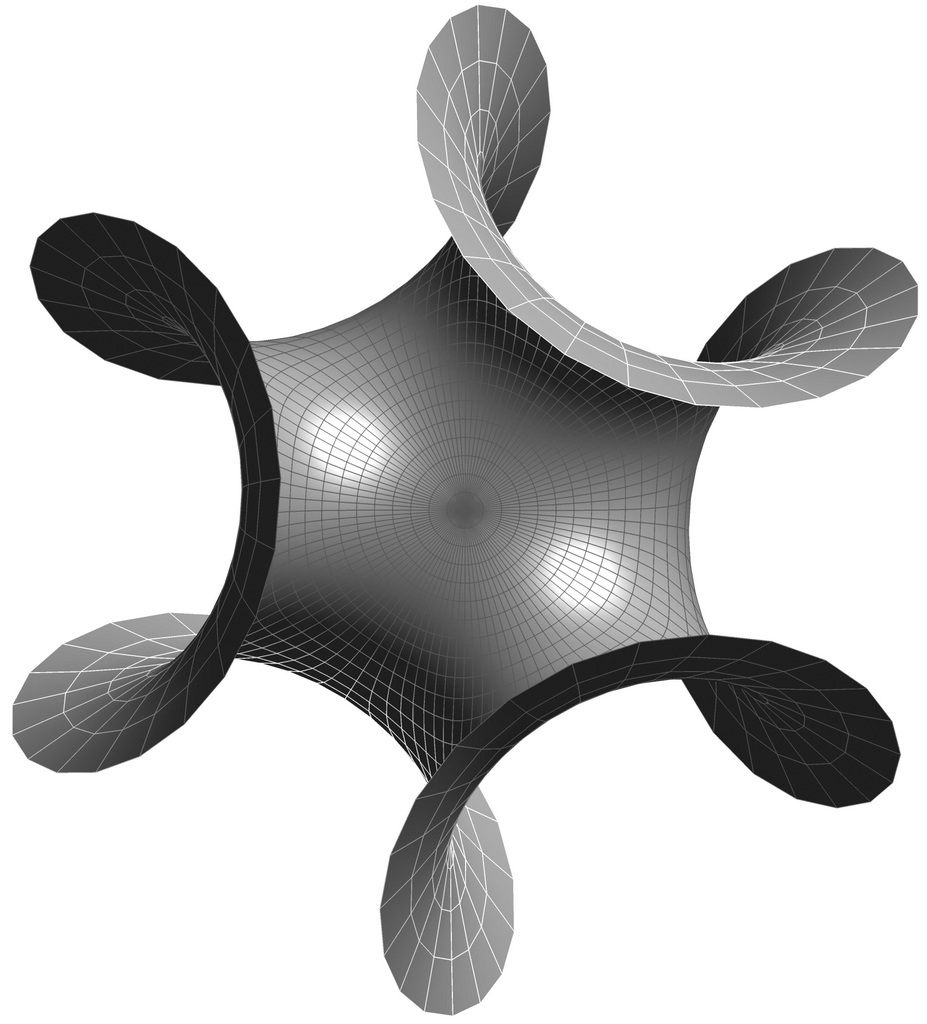} \quad & \quad
\includegraphics[height=55mm]{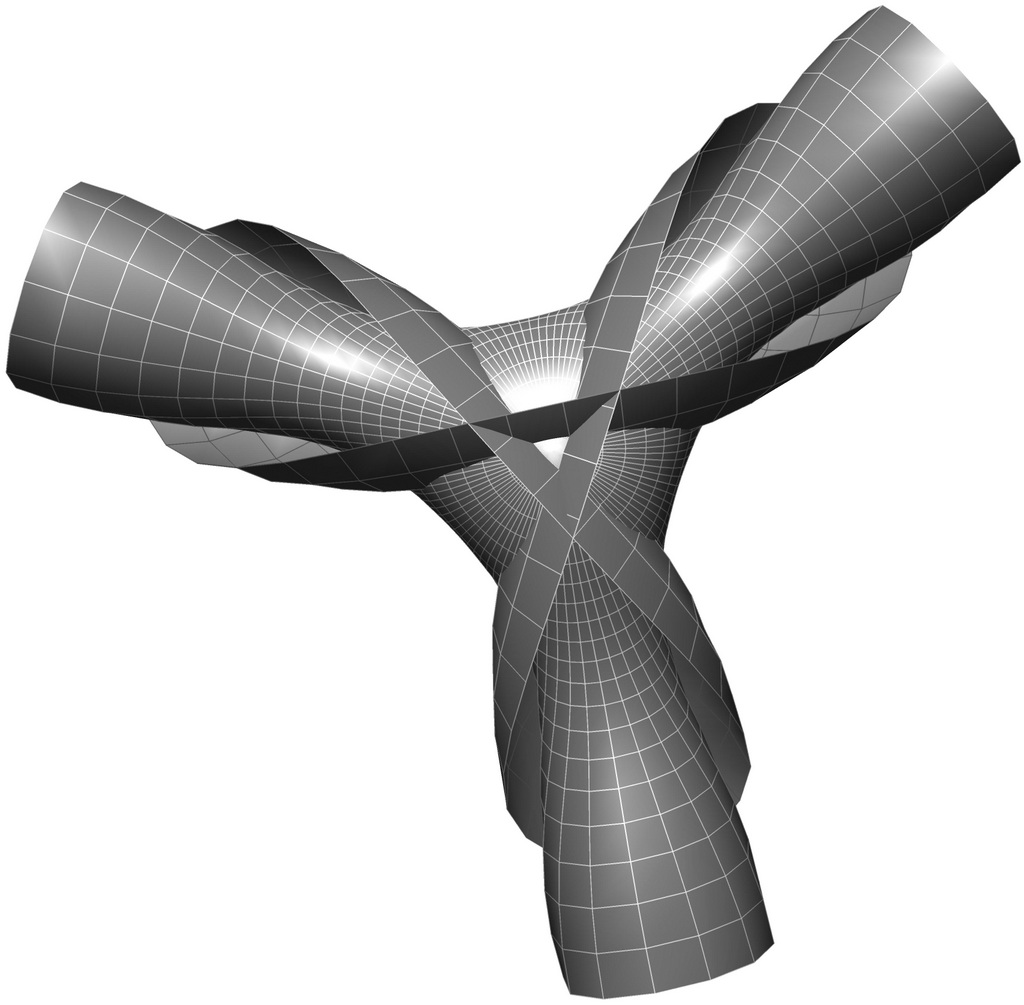}  
 \vspace{2ex} \\
 H=10^{-8}, & 
 H=1   
\end{array}
$
\caption{Left: Almost minimal version of Enneper's surface of order 2. Right: 3-legged Mr. Bubble, or Smyth surface. } 
\label{figuresmyth}
\end{figure}

Given a CMC surface or minimal surface, although the associated family depends on the choice of the basepoint, there is sometimes a natural such choice, and therefore a canonical family associated.  For example, Enneper's surface of order $k$ (See Example \ref{enneperexample} below)  has a finite order rotational symmetry about a central point.  Such a symmetry is preserved under the deformation if the basepoint is chosen to be this central point, and therefore, making this choice, there is one natural family of  CMC surfaces that have this symmetry and which includes Enneper's surface at $h=0$.  As we will see below, these turn out to be Smyth surfaces, studied in \cite{smyth}. Solutions, computed numerically for $H=10^{-8}$ and $H=1$, are shown in Figure \ref{figuresmyth}.\\

\subsection{The dressing action on minimal surfaces}
The dressing action is a group action on the space of solutions, which
generally exists for any 
integrable system represented by maps into loop groups.  It was introduced into the
study of harmonic maps by Uhlenbeck \cite{Uhl}. See Wu \cite{wu1997} for a description
for the case of CMC surfaces.
Dressing can be described as an action on the normalized potential $\hat \eta$.
A minimal surface also has a normalized potential, but, unlike in the non-minimal
case, the correspondence is not bijective: there are many minimal surfaces with
the same normalized potential.   In \cite{dorfmeisterpedittoda}, the dressing
 action is defined, via these normalized potentials, on minimal surfaces, giving
 an action on the set of equivalence classes of minimal surfaces with the same
 potential.  In Section \ref{dressingsection}, we use the analysis of Wu \cite{wu1997}
 to determine the class of dressing elements that are independent of $h$, in
 Theorem \ref{dressingthm1}.  A corollary of this, together with
 Theorem \ref{mainthm}, is that the dressing action defined in \cite{dorfmeisterpedittoda}  in fact gives a well-defined group action on
 the space of minimal immersions. This is Theorem \ref{dressingthm2}.\\

\subsection{Properties preserved under the deformation}
The classical Weierstrass data for many minimal surfaces is known. Therefore, Theorem \ref{mainthm} can easily be used to construct examples of non-minimal CMC surfaces, with the generalized Weierstrass data given explicitly by (\ref{munupotential}).  Clearly, it is of interest to 
know what properties are preserved as the mean curvature $h$ varies. \\

Global topological properties are not preserved: it is true that any minimal surface can be represented by Weierstrass data on a contractible domain (the universal cover); and the same 
holds for any non-minimal CMC surface other than the round sphere.  However, in general, any closing properties of the surface will be lost as $h$ varies.  \\

Because the Hopf differential is preserved, it follows (see Remark \ref{pcremark}) that, not
only umbilic points, but also 
principal curves in the coordinate domain are preserved under the deformation. Moreover, (see Remark \ref{pcremark2}),
the values of the principal curvatures at the basepoint are given explicitly by 
$\kappa_\pm(z_0) = \pm \kappa_0 + h$, where $\pm \kappa_0$ are the principal curvatures at $z_0$ 
for the minimal surface in the family. This gives a local picture of the deformation around the basepoint.\\

In Section \ref{symmetrysection} we investigate  surfaces with symmetries.  We consider surfaces which have a reflection symmetry about a plane and surfaces with a finite order rotational symmetry about a point in the surface.  In the first case, if the plane of symmetry contains the basepoint, then we show that the surface has such a symmetry if and only if coordinates can be chosen such that the Weierstrass data are real-valued along the real line.  In the second case, if the rotation point is the basepoint $z_0$, then we show that
the surface has the symmetry if and only if the Weierstrass data have
Laurent expansions including only certain powers of $z$.   Consequently, Theorem \ref{symmetrythm} states that such symmetries are
preserved under this deformation, provided the basepoint is chosen appropriately.


\section{The loop group formulation and DPW method}  \label{dpwsection}
In this section we summarize well known facts about CMC surfaces and their 
construction via integrable systems methods. The notation and conventions are the
same as those used in \cite{bjorling}, where more details and references can be found.\\

\subsection{The loop group characterization of CMC maps}
Let $\Sigma$ be a contractible Riemann surface, and suppose $f: \Sigma \to \E^3$ is
a conformal immersion with mean curvature $H$. 
Choosing conformal coordinates $z = x + iy$,
a function $u: \Sigma \to \real$ is defined by the expression
$\dd s^2 = 4e^{2u}(\dd x^2 + \dd y^2)$ for the induced metric.  The matrices for the first
and second fundamental forms $I$ and $II$, with respect to the coordinates $x$,$y$ are then:
\beq \label{fundamentalforms}
I = \bbar 4e^{2u} & 0 \\ 0 & 4e^{2u} \ebar, \quad \quad
II = \bbar 4H e^{2u} + Q + \bar Q & i(Q-\bar Q) \\
     i(Q-\bar Q) & 4He^{2u} -(Q+\bar Q) \ebar,
\eeq
\\
where $H := e^{-2u}\langle f_{xx} + f_{yy}, N \rangle/8$ is
the mean curvature, and 
$Q := \langle N,f_{zz} \rangle$.
The differential $2$-form 
 $Q \, \dd z^2$ is called the \emph{Hopf differential}.\\

The Lie algebra $\mathfrak{su}(2)$, is identified with $\E^3$ via the
following  basis, which is orthonormal with respect to the
 inner product $\langle X,Y \rangle = -  \text{Trace} (XY)/2$:
\bdm
e_1 = \bbar 0 & -i \\ -i & 0 \ebar, \hspace{1cm}
e_2 = \bbar 0 & 1 \\ -1 & 0 \ebar, \hspace{1cm}
e_3 =  \bbar i & 0 \\ 0 & -i \ebar.
\edm
Given a choice of unit normal $N$, 
the \emph{coordinate frame} $F: \Sigma \to SU(2)$ is uniquely determined (up to sign) by the conditions
\beq \label{framedef}
\Ad_F e_1  = \frac{f_x}{|f_x|}, \quad
\Ad_F e_2  = \frac{f_y}{|f_y|}, \quad \Ad_F e_3 = N.\\
\eeq

Differentiating the expressions
$f_z = e^u \Ad_F(e_1-ie_2)$ and $f_{\bar z} = e^u \Ad_F(e_1+ie_2)$,
 one obtains  the following expression for the
  connection coefficients $U := F^{-1}F_z$ and 
$V := F^{-1}F_{\bar{z}}$:
\begin{equation}\label{UhatandVhat}
U = \frac{1}{2} \begin{pmatrix} u_z & -2 H e^u \\ 
                                   Q e^{-u}  & -u_z \end{pmatrix} 
, \hspace{1cm} V = \frac{1}{2} \begin{pmatrix} -u_{\bar z} & - \bar Q e^{-u}\\ 
                            2  H e^u & u_{\bar z} \end{pmatrix}. \end{equation}

If $H$ is constant, we can extend the frame $F$ to a loop group valued map
$\hat F$ as follows: first extend the Maurer-Cartan form of $F$
to a loop-algebra valued $1$-form
$\hat \alpha := \hat U  \dd z + \hat V \dd \bar{z}$, where
\begin{equation}\label{withlambda}
\hat U = \frac{1}{2} \begin{pmatrix} u_z & -2 H e^u \lambda^{-1} \\ 
                                   Q e^{-u} \lambda^{-1}  & -u_z \end{pmatrix} 
, \hspace{1cm} \hat V = \frac{1}{2} \begin{pmatrix} -u_{\bar z} & - \bar Q e^{-u} \lambda\\ 
                            2  H e^u \lambda & u_{\bar z}  \end{pmatrix}. 
\end{equation}
The $1$-form $\hat \alpha$ satisfies the Maurer-Cartan equation
$\dd \hat \alpha + \hat \alpha \wedge \hat \alpha = 0$
 for all $\lambda \in \C \setminus \{ 0 \}$ 
if and only if the mean curvature $H$ is constant or, equivalently, the
Hopf differential is holomorphic.\\

Now fix a basepoint $z_0 \in \Sigma$ and set $E_0 :=  F(z_0)$. We extend the initial condition $E_0$ to a twisted loop $\hat E_0$ by the formula
\beq    \label{fhatzero}
\hat E_0 = \bbar A_0 & \lambda B_0 \\ -\lambda^{-1} \bar B_0 & \bar A_0 \ebar,
\quad \quad \textup{where } E_0 = \bbar A_0 & B_0 \\ -\bar B_0 & \bar A_0 \ebar.
\eeq
Now integrating $\hat \alpha$ with the initial condition $\hat F(z_0) = \hat E_0$,
one obtains the \emph{extended frame} $\hat F:  \Sigma \to \uu$,
a map into the twisted group of loops in $G:= SU(2)$.\\

If we denote, for a $\Lambda G^\C$-valued map $\hat X$, the corresponding map into
the group $G^\C$, obtained by evaluating at the loop value $\lambda =1$,
 by $X = \hat X |_{\lambda =1}$, then the above notation for $F$ and $\hat F$
 is consistent.\\

If $h$ is any nonzero real number, and $\lambda_0 \in \SSS^1$,
the Sym-Bobenko formula is:
\beq \label{symformula}
 \sym_{h,\lambda_0}(\hat F) :=  -\frac{1}{2h} \left. \left(  2 i \lambda 
\partial_\lambda \hat F \,  \hat F^{-1} \ + 
\, \hat F e_3 \hat F^{-1} - e_3 \right)\right |_{\lambda = \lambda_0}. 
\eeq
Note that if $\hat F: \Sigma \to \uu$ is a smooth map,
 then $\sym_{h,\lambda}(\hat F): \Sigma \to \mathfrak{su}(2) = \E^3$ is 
 also smooth.  We will mainly use the formula for the case $\lambda_0 = 1$, and
 therefore use the notation $\sym_h(\hat F) = \sym_{h,1}(\hat F)$.\\


If $\hat F$ is an extended coordinate frame for a CMC $H$ surface, and $H \neq 0$,
then $f$ is retrieved by the formula
\beq  \label{normalizedsym}
f(z) =  \mathcal{S}_H (\hat F(z)) +  f(z_0).
\eeq
 Moreover, the Sym-Bobenko formula is invariant under right multiplication by a diagonal
unitary matrix valued function. This corresponds to a change of $SU(2)$ frame
for the Gauss map.  Hence there is a well defined lift 
$[\hat F]: \Sigma  \to \uu/K$, where $K=G_\sigma$ is the diagonal subgroup, of any 
CMC $H$ immersion $f$ of a simply connected 
surface, independent of coordinates and choice of frame; and, if
$H \neq 0$, the
formula $\sym_H([\hat F])$ is well defined and gives $f$ up to a translation. 
The case $H=0$ will be discussed below.\\

Slightly more generally, define an \emph{admissible frame} to be any smooth map $\hat F$,
from a Riemann surface  $\Sigma$ into $\uu$ with the property that the
Maurer-Cartan form $\hat F^{-1} \dd \hat F$  is a Laurent polynomial
of the form $\hat \alpha = \alpha_{-1} \lambda^{-1} + \alpha_0 + \alpha_1 \lambda$
where the $(0,1)$ part of $\alpha_{-1}$ is zero.  The reality condition and twisting
on $\uu$ mean we can write
\bdm
\hat \alpha = A \lambda^{-1} \dd z + \alpha_0 + \bar A \lambda \dd \bar z,
\edm
where $A$ is an off-diagonal $\mathfrak{su}(2)$-valued function and 
$\alpha_0$ is a diagonal
$\mathfrak{su}(2)$-valued 1-form. The admissible frame is \emph{regular} if 
the upper right component $A_{12}$ is non-vanishing. 
For $H \neq 0$, the extended coordinate frame described above is a regular admissible frame.\\

Finally, for \emph{any} regular admissible frame $\hat F$ and any value of
$(h,\lambda) \in \real^* \times \SSS^1$, the map 
$f = \sym_{h,\lambda}(\hat F)$ is a conformal CMC $h$ immersion into $\real^3$.\\


\subsection{The DPW construction}\label{dpwmethodsect}
Let $\uc$ denote the group of twisted loops in $G^\C = SL(2,\C)$ 
and  $\Lambda^+ G_\sigma^\C$  and $\Lambda^- G_\sigma^\C$
the subgroups of loops which extend holomorphically to the unit disc and
the exterior disc $\{\lambda ~|~ |\lambda|>1\}$ in the Riemann sphere respectively.
For the purpose of normalizations, we also use the subgroups
\beqas
 \Lambda^-_* G^\C_\sigma := \{ B \in \Lambda^-G^\C_\sigma ~|~ B(\infty) = I\},\\
 \ustar := \{ B \in \Lambda^+G^\C_\sigma ~|~ B(0) = \textup{diag}(\rho , \rho^{-1}), ~\rho \in \real, ~\rho>0 \}.
\eeqas
The \emph{Birkhoff decomposition} \cite{PreS} states that any $g$ in a certain open dense subset 
of $\uc$ (called the \emph{big cell}) has a unique factorization
\beq \label{birkhoff}
g = g_- g_+, \quad \quad g_- \in \Lambda^-_* G^\C_\sigma, \,\,\, g_+ \in \Lambda^+ G^\C_\sigma.
\eeq
The \emph{Iwasawa decomposition} \cite{PreS}  states that
any $g$ in $\uc$ can be uniquely expressed as a product
\beq \label{iwasawa}
g = FB, \quad \quad F \in \uu, \,\,\, B \in \ustar. 
\eeq
In both decompositions, the factors on the right hand side depend real analytically on $g$.
If one takes $g_- \in \Lambda^- G^\C_\sigma$ instead of $\Lambda^-_* G^\C_\sigma$,
and $B \in \Lambda^+ G^\C_\sigma$ instead of in $\ustar$, then the factors in the 
decompositions are only unique up to a middle term which is a constant loop.\\

A brief version of the DPW method (see \cite{DorPW, DorH:cyl}) states the
following:  let $\hat F: \Sigma  \to \uu$ be an  
extended frame for a non-minimal CMC $H$ immersion $f: \Sigma \to \E^3$, where $\Sigma$ is
a contractible Riemann surface. Assume $\hat F(z_0) = \hat E_0$, of the form
(\ref{fhatzero}), at some 
 fixed basepoint $z_0$. The coordinate frame $\hat F$ is uniquely determined by $z_0$ and
 $E_0$.  Hence a unique meromorphic map $\hat \Phi: \Sigma \to \uc$
 is defined by the normalized Birkhoff decomposition, performed pointwise over the 
 pre-image $\Sigma^\circ := \{ z \in \Sigma ~|~\hat E_0^{-1} \hat F (z) \in \Lambda^-_* G^\C_\sigma \cdot \Lambda^+ G^\C_\sigma  \}$
 of the big cell:
 \bdm
\hat E_0^{-1} \hat F = \hat \Phi \hat G_+, \quad \quad \hat \Phi(z) \in \Lambda^-_* G^\C_\sigma, \quad
   \hat G_+ (z) \in \Lambda^+ G^\C_\sigma. 
  \edm
For the rest of this section and the next, we take $\hat E_0 =I$, to simplify the expressions. This has no effect on the geometry - a change of $\hat E_0$ amounts to an isometry of the ambient space $\E^3$.\\

 Note that it is simple to check that $\hat \Phi$ is holomorphic on $\Sigma^\circ$, and it is shown in \cite{DorPW} that this map has only poles at the boundary
 of this open dense set.  Moreover, in conformal coordinates $z=x+iy$, the Maurer-Cartan
 form of $\hat \Phi$ has the form:
 \bdm
 \hat \eta = \hat \Phi^{-1} \dd \hat \Phi = 
    \bbar  0 & -\frac{H}{2}a \\ \frac{Q}{a} & 0 \ebar \lambda^{-1} \dd z,
  \edm
 where $a(z)$ is meromorphic and $Q(z) \dd z^2$ is the (holomorphic) Hopf differential of $f$.
 The $1$-form $\hat \eta$ is called a \emph{normalized potential} for $f$,
 and $\hat \Phi$ is called a \emph{normalized meromorphic frame}.  An explicit formula for
 the normalized potential, in terms of the metric and Hopf differential, is
 given by Wu in \cite{wu1999}.\\
 
 Conversely,  given a pair  of functions $(a, Q)$, with
  $a$ meromorphic and $Q$ holomorphic, if the zeros and poles 
   have certain sufficient (and necessary) 
   conditions, then the formula above for $\hat \eta$
  is a meromorphic potential for a  CMC surface \cite{DH97}. The surface is uniquely
  determined by $\hat \eta$ and the basepoint $z_0$.
  The most straightforward condition on the poles and zeros of $\hat \eta$ 
  is that $a$ is holomorphic and non-vanishing, but in 
 general one has
 \begin{theorem} \cite{DH97}  \label{zeropoletheorem}
  Let $\Ord(Q)$ denote the vanishing order of $Q$ at a point. Then necessary and 
 sufficient conditions for $\hat \eta$ to correspond to a smooth surface around the point are:
 \begin{enumerate}
  \item If $\Ord(a)<0$  then $\Ord(a)=-2$ or, for some integer $r \geq 1$, either
  \bdm
  \Ord(Q)=\frac{-\Ord(a)}{2r}-2, \quad \textup{or} \quad \Ord(Q)= \frac{-\Ord(a)-2}{2r}-2;
  \edm
  \item If $\Ord(a)>0$, then, for some integer $r \geq 1$,
  \bdm
  \Ord(Q)=\frac{\Ord(a)}{2r}-2, \quad  \textup{or} \quad \Ord(Q)= \frac{\Ord(a)+2}{2r}-2.
  \edm
 \end{enumerate}
\end{theorem}
Note that the above conditions also ensure that the potential is meromorphically integrable, because they rule out the possibility that
$\hat \eta$ has a pole of order 1. 
We further remark that if $a$ has a zero  and $\Ord(Q) \geq \Ord(a)$, so that $\hat \eta$ is holomorphic
at the point, then the surface has a branch point.\\

If $\hat \eta$ is \emph{holomorphic}, then a frame $\hat F$ is recovered 
  as follows:  solve the equation
  $\dd \hat \Phi = \hat \Phi \hat \eta$, with $\hat \Phi(z_0)= I$.
  For each $z$ perform the unique Iwasawa 
  decomposition
\bdm
\hat \Phi =  \hat F \, \hat B_+, \quad \quad \hat F(z) \in \uu, \quad \hat B_+(z) \in \ustar. 
\edm
  Then $\hat F$ is  an extended frame  for a CMC $H$ surface $f = \sym_H(\hat F)$.  
This is not the coordinate frame for  $f$ in general 
(assuming $a$ is non-vanishing so that $f$ is immersed),
 but represents the same map
$[\hat F]: \Sigma \to \uu/K$, and therefore the same surface.  \\

If $\hat \eta$ has poles, then one can prove that the surface $f$ is also immersed at a pole of $\hat \eta$. To do this, one needs to perform dressing first
(see \cite{DH97}).\\

\section{From  CMC surfaces to minimal surfaces}
Let $\Sigma$ be a contractible Riemann surface, and $f_H: \Sigma \to \E^3$ a
conformal immersion of constant mean curvature $H \neq 0$.  Choose a base point
$z_0$, and coordinates for $\E^3$ so that the coordinate frame satisfies $F(z_0)=E_0=I$.
There is associated a unique normalized potential $\hat \eta$, as
described above.  Let us now 
consider $H$ as a real parameter, and write
\bdm
 \hat \eta_h = \hat \Phi^{-1} \dd \hat \Phi = 
    \bbar  0 & -\frac{h}{2}a \\ \frac{Q}{a} & 0 \ebar \lambda^{-1} \dd z,
  \edm
and denote by $\hat \Phi_h$ the associated normalized meromorphic frame with
$\hat \Phi_h(z_0)=I$.  \\

The conditions on the zeros and poles of the potential for it to correspond to a 
smooth surface (Theorem \ref{zeropoletheorem}) do not depend on $h$, but only on $a$ and $Q$.
These conditions are necessarily satisfied, since $\hat \eta_H$ came from an 
immersed CMC surface. Hence, there is a 
smooth CMC immersion $f_h: \Sigma \to \E^3$ corresponding to $\hat \eta_h$ 
for every $h \neq 0$.  Given the choice of basepoint $z_0$, the initial conditions 
$\hat F_{C,h}(z_0)=I$ for the coordinate frame,
 and $f_h(z_0)=f_H(z_0)$ for the surface, $f_h$ is uniquely determined by the formula
\bdm
f_h(z) = \mathcal{S}_h (\hat F_{C,h}(z)) +  f_H(z_0). 
\edm
If we restrict to the preimage of the big cell, we also have a unique Birkhoff decomposition
\bdm
 \hat F_{C,h} = \hat \Phi_h \hat G_{h,+},  \quad \quad \hat G_{h,+} \in \Lambda^+G_\sigma^\C.
\edm

Thus we have a family $f_h: \Sigma \to \E^3$ of
 immersed surfaces of constant mean curvature $h$, all  with the same Hopf differential,
$Q \dd z^2$, and such that $f_h$ and $f_H$, together with their tangent planes,
agree at $z_0$. Moreover, the family $f_h$ depends
real analytically on $h$, because $h$ appears analytically in the data and all the
operations performed to obtain $f_h$ preserve this property.  We now show that this family 
includes the value $h=0$:

\begin{theorem} \label{thm1} Let $f_H$ be as above, with extended coordinate frame $\hat F_C$.
Let $\Sigma^\circ$ denote  the pre-image under $\hat F_C$ of the
big cell, i.e. the open dense set 
\bdm
\Sigma^\circ := \{z \in \Sigma ~|~ \hat F_C(z) \in 
\Lambda^-_* G^\C_\sigma \cdot \Lambda^+ G^\C_\sigma \}.
\edm 

\begin{enumerate}
\item  \label{thm1item1}
The map $\mathcal{F}: \Sigma \times \real^* \to \E^3$, given by $\mathcal{F}(z,h) = f_h(z)$,
extends to a real analytic map $\Sigma \times \real \to \E^3$. \\

\item  \label{thm1item2}
The map 
$f_0\big|_{\Sigma^\circ}: \Sigma^\circ \to \E^3$ given by 
restricting  $\mathcal{F}(z,0)$ to $\Sigma^\circ$ is a conformally immersed minimal 
surface with Hopf differential $Q \dd z^2$ and metric given by
\beq   \label{f0metric}
\dd s^2 = (1+|g|^2)^2 |a|^2(\dd x^2 + \dd y^2), \quad \quad g(z) = \int_{z_0}^z \frac{Q(\tau)}{a(\tau)} \dd \tau.
\eeq
The map $f_0$,
 together with its tangent plane, agrees with $f_H$ at $z_0$.\\
 \end{enumerate}
\end{theorem}

\begin{proof}

\textbf{Item \ref{thm1item1}:} The idea of the argument is to get two expressions for $f_h(z)$:
one that is very explicit, but only defined on an open dense set;  and  another that is not so explicit, but defined everywhere.\\

 We give the argument first on $\Sigma^\circ$, which is an 
open dense set, and then extend to $\Sigma$.  On $\Sigma^\circ$
we can use, instead of the coordinate frame, the unique smooth frame $\hat F_h$
given by the following Iwasawa decomposition 
\beq  \label{iwasawa1}
\hat \Phi_h =  \hat F_h \hat B_{h,+}, \quad \hat B_{h,+}(z) \in \Lambda^+_P G^\C_\sigma.
\eeq
Note that $\mathcal{S}_h(\hat F_h)$ has the same value whether one uses this frame
or the coordinate frame.  This frame has the advantage that it can be computed 
explicitly at $h=0$, where the potential
\bdm
\hat \eta_0 = \bbar 0 & 0 \\ \frac{Q}{a} & 0 \ebar \lambda^{-1} \dd z
\edm
can be integrated to obtain
\bdm
\hat \Phi_0 = \bbar 1 & 0 \\ \lambda^{-1} g & 1 \ebar, \quad
     g(z) = \int_{z_0}^z \frac{Q(\tau)}{a(\tau)} \dd \tau.
\edm
The  Iwasawa decomposition (\ref{iwasawa1}) is
\beq  \label{phizerodecomp}
\hat \Phi_0 =  \hat F_0 \hat B_+, \quad
     \hat F_0 = \dfrac{1}{\sqrt{1+|g|^2}}  \bbar 1 & - \lambda \bar g \\ 
               \lambda^{-1} g & 1 \ebar, \quad \hat B_+(z) \in \Lambda^+_P G^\C_\sigma.
\eeq

Now $f_h$, which is real analytic in $h$ on $\real^*$, is given by
\beqas
 f_h &=&  \sym_h(\hat F_h) +  f_H(z_0) \\
 &=& 
-\frac{1}{2h} \left. \left(  2 i \lambda 
\frac{\partial \hat F_h}{\partial \lambda} \,  \hat F_h^{-1} \ + 
\, \Ad_{\hat F_h} e_3   - e_3 \right)\right |_{\lambda = 1} +   f_H(z_0).
\eeqas
$f_H(z_0)$ is constant, so we only need consider the first term.
By construction, the expression inside the parentheses is analytic for all
$h$, and therefore has an expansion in $h$, around $h=0$, given by
\bdm
\left. \left(  2 i \lambda  \frac{\partial \hat F_h}{\partial \lambda} \,  \hat F_h^{-1} \ + 
\, \Ad_{\hat F_h} e_3  - e_3 \right)\right |_{\lambda = 1} 
= C_0  + O(h).
\edm
Analyticity of $f_h$ at $h=0$ will follow if we can show that $C_0 = 0$.
But $\hat F_h = \hat F_0 + O(h)$, and so 
\beq    \label{Czeroexpression}
C_0 = \left. \left(  2 i \lambda  \frac{\partial \hat F_0}{\partial \lambda} \,  \hat F_0^{-1} \ + 
\, \Ad_{\hat F_0} e_3  - e_3 \right)\right |_{\lambda = 1}.
\eeq
It is easy to verify that this expression is zero for any loop $\hat F_0$ in 
$\Lambda G_\sigma$ of the form 
\bdm
\bbar A & -\lambda \bar B \\ \lambda^{-1} B & \bar A \ebar,
\edm
and $\hat F_0$, given at (\ref{phizerodecomp}), is indeed of this form. Thus, $C_0=0$.\\

Finally we must consider points $z$ on the boundary of $\Sigma^\circ$, that is, points where $\hat \eta$ has a pole.
In this case let us consider the extended coordinate frame $\hat F_{C,h}$. For $h\neq 0$, this frame is well defined on the whole of $\Sigma$, because it is constructed from the coordinate frame of a smooth surface. Moreover, it is analytic in all parameters.  On the pre-image of the big cell,
  $\Sigma^\circ$, we have the relation
 \bdm
 \hat F_{C,h} = \hat F_h \bbar \mu & 0 \\ 0 & \bar \mu \ebar,
 \edm
 for some unitary function $\mu$, constant in $\lambda$. Hence, for $z \in \Sigma^\circ$,
 one obtains:
 \beqas
 \left. \left(2 i \lambda  \partial_\lambda \hat F_{C,h} \,  \hat F_{C,h}^{-1} \ + 
\, \Ad_{\hat F_{C,h}} e_3  - e_3 \right) \right|_{\lambda=1}  &=& 
  \left. \left(2 i \lambda  \frac{\partial \hat F_h}{\partial \lambda} \,  \hat F_h^{-1} \ + 
\, \Ad_{\hat F_h} e_3  - e_3 \right)  \right|_{\lambda=1} \\
 &=& C_0 + O(h).
\eeqas
We have already shown that $C_0(z) = 0$ for  $z \in \Sigma^\circ$. Since 
$C_0$ is continuous (in fact analytic) in $z$, and $\Sigma^\circ$ is open 
and dense, it must vanish everywhere.
  By a similar argument to that given above on $\Sigma^\circ$, applied now to the expression involving
  $\hat F_{C,h}$ on $\Sigma$,
it follows that, for all $z$, the map $f_h(z)$ is analytic in $h$ at  $h=0$.\\

\textbf{Item \ref{thm1item2}:}

First note that the extended coordinate frame $\hat F_{C,h}$, for $h \neq 0$, is in the big cell for all $z \in \Sigma^\circ$.
Otherwise, $\hat \eta_h$ would have a pole on $\Sigma^\circ$: but this condition is independent of 
$h$, and $\hat \eta_H$ has no poles on this set.
Now, Birkhoff decomposing $\hat F_{C,h}$ pointwise we have
$\hat F_{C,h} =  \hat \Phi_h \hat H_{+,h}$,  for a unique real-analytic
$\Lambda^+ G_\sigma^\C$-valued map  $\hat H_{+,h}$, which has a Fourier
expansion in $\lambda$ of the form
\bdm
\hat H_{+,h} = \bbar \rho_h^{-1} & 0 \\ 0 & \rho_h \ebar + O(\lambda).
\edm
It follows that $\hat U_h := \hat F^{-1}_{C,h}  (\hat F_{C,h})_z$ is
of the form
\bdm
\hat U_h = \bbar * & - \rho^2_h \dfrac{h a}{2} \lambda^{-1} \\ \rho^{-2}_h \dfrac{Q}{a} \lambda^{-1} & * \ebar.
\edm
Differentiating the formula $f_h(z) = \sym_h(\hat F_{C,h}(z))+ f(z_0)$, we
obtain
\bdm
\frac{\dd f_h}{\dd z} =  \frac{1}{2} a \rho_h^2 \Ad_{F_{C,h}}(e_1-i e_2).
\edm
Similarly, since the reality condition on the loop group means that
 $\hat V_h := \hat F^{-1}_{C,h}  (\hat F{C,h})_{\bar z} =  -\overline{\hat U_h}^t$ one computes 
 $\frac{\dd f_h}{\dd \bar z} = \frac{1}{2} a \rho_h^2 \Ad_{F_{C,h}}(e_1+i e_2)$,
so that
\bdm
\frac{\dd f_h}{\dd x} =  a \rho_h^2 \Ad_{F_{C,h}}(e_1), \quad \quad
\frac{\dd f_h}{\dd y} =  a \rho_h^2 \Ad_{F_{C,h}}(e_2).
\edm
We want to take the limit as $h \to 0$. For this consider the normalized
 frame $\hat F_h$, given by (\ref{iwasawa1}).  This has the Birkhoff
 decomposition 
 $\hat F_h = \hat E_0 \hat \Phi_h \hat B_{h,+}^{-1}$, with the Fourier expansion
 \bdm
 \hat B_{h,+}^{-1} = \bbar \tilde \rho_h^{-1} & 0 \\ 0 & \tilde \rho_h \ebar + O(\lambda), \quad  \quad \tilde \rho_h(z) \in \real_{>0}.
 \edm
 Since the factor $\hat X \in \Lambda G_\sigma$ of any Iwasawa decomposition $\hat \Phi_h = \hat X \hat B_+$ is unique up to right multiplication by a constant (in $\lambda$) matrix,  the values
  $\hat B_{h,+}^{-1}(z)$ and $\hat H_{h,+}(z)$ are necessarily related by left
 multiplication by such a matrix, and we have
 \bdm
 \rho(z) = \tilde \rho(z) e^{i \theta_h(z)}, \quad \quad \theta_h(z) \in \real.
 \edm
Thus we have
\bdm 
\left\| \frac{\partial f_h}{\partial x}\right\| 
= \left\| \frac{\partial f_h}{\partial y}\right\| =    \tilde \rho^2_h |a|. 
\edm
Now by the explicit Iwasawa decomposition (\ref{phizerodecomp}) of
$\hat \Phi_0$, we have, at $h=0$,
\bdm
\tilde \rho_0 = \sqrt{1+|g|^2}, \quad  \quad g(z) = \int_{z_0}^z \frac{Q(\tau)}{a(\tau)} \dd \tau.
\edm
It follows that 
\beq \label{limitingmetric}
\left \| \frac{\partial f_0}{\partial x}\right\| = \left\| \frac{\partial f_0}{ \partial y}\right\| = (1+|g|^2) |a|.
\eeq
Together with the fact that $\partial_x f_h$ and $\partial_y f_h$ are
orthogonal for all $h \neq 0$,   this implies that 
 $f_0$ is conformally immersed on $\Sigma^\circ$, with metric given by
 (\ref{f0metric}).
  The fact that the mean curvature
is zero and the Hopf differential is $Q \dd z^2$ follow by continuity with
respect to the parameter $h$.\\
\end{proof}

\begin{remark} The above proof shows, in fact, that $f_0$ is conformally immersed on the set where
$(1+|g|^2)|a|$ is finite and non-vanishing.  This set includes $\Sigma^\circ$, but would be larger in general.
\end{remark}




\section{From minimal surfaces to  CMC surfaces}
The connection between the loop group formulation for CMC surfaces and the classical
Weierstrass representation for minimal surfaces was investigated by Dorfmeister, Pedit and Toda \cite{dorfmeisterpedittoda}.  More details of the setup used here can be found (with slightly different conventions) in that reference.\\

\subsection{The classical Weierstrass representation in terms of the $SU(2)$-frame}
\label{classicalwsect}
Choosing the symmetric space representation $\SSS^2 = SU(2)/\SSS^1$, where $\SSS^1$ is
the diagonal subgroup, and the complex structure given by 
\bdm
\mathfrak{p}^\C = \C \bbar 0 & 0 \\ 1 & 0\ebar \oplus \C \bbar 0 & 1 \\ 0 & 0 \ebar =
T_0^{(1,0)} \SSS^2 \oplus T_0^{(0,1)} \SSS^2,
\edm
a map $g: \Sigma \to \SSS^2$ is holomorphic if and only if any lift $F$ into $SU(2)$
satisfies $F^{-1} \dd F = \alpha_\mathfrak{p}^\prime + \alpha_\mathfrak{k} + \alpha_\mathfrak{p}^{\prime \prime}$, with $\alpha_\mathfrak{p}^\prime$ 
a  $T_0^{(1,0)} \SSS^2$-valued 1-form.    Hence the expression (\ref{UhatandVhat})
shows that a conformally immersed surface in $\E^3$ is minimal if and only if
its Gauss map is holomorphic.\\

The classical Weierstrass representation for the surface is obtained as follows:
if $F_C$ is  the coordinate frame defined by (\ref{framedef}), we can write
\bdm
F_C = \bbar A & B \\ -\bar B & \bar A\ebar, \quad \quad 
F_C^{-1} \partial _z F_C = \frac{1}{2} \begin{pmatrix} u_z & 0 \\ 
                                   Q e^{-u}  & -u_z \end{pmatrix}.
\edm
We deduce from this that 
\bdm
A = e^{-u/2}s, \quad B = e^{-u/2} \bar r,
\edm
for a pair of  holomorphic functions $s$ and $r$ which satisfy
\bdm
 e^u = s \bar s + r \bar r, \quad \quad
Q = 2(rs_z - s r_z).
\edm

Since $F_C$ is the coordinate frame, 
we have $f_z = e^u \Ad_F(e_1-ie_2)$, which works out to
\beq  \label{fzrs}
f_z = (s^2-r^2)\, e_1 -i(s^2+r^2)\, e_2 - 2sr \, e_3,
\eeq
and the Weierstrass representation is
\beq  \label{wformula}
f = 2 \Re \int_{z_0}^z f_z \dd z.
\eeq
The commonly used Weierstrass representation states that, given a
holomorphic function $\mu$ and a meromorphic function $\nu$ on $\Sigma$,
such that $\mu \nu^2$ is holomorphic, then a minimal surface $f: \Sigma \to \real^3$ is 
given by the above integral, with
\beq  \label{classicalw}
f_z =  \mu(1-\nu^2)e_1 -  i \mu (1+\nu^2) \, e_2 -  2 \mu \nu \, e_3.
\eeq
(The conventions used here are chosen here for convenience).
The surface is regular at points where either
\begin{enumerate}
\item $\Ord(\mu) = 0$ and $\Ord(\nu) \geq 0$,  or
\item $0 \leq \Ord(\mu) = -2 \Ord(\nu)$.
\end{enumerate}

Comparing this with our data, we have:   
$s^2 = \mu$, $r^2 = \mu \nu^2$, $sr = \mu \nu$.
Thus, given classical Weierstrass data $\mu$ and $\nu$, the coordinate frame above is
well defined on the set $\Sigma$:  the function $\mu$ has zeros only of even order
and, up to an irrelevant sign, we can solve for
\beq  \label{srmunu}
s= \sqrt{\mu} , \quad \quad r=  \nu \sqrt{\mu},
\eeq
and the metric and Hopf differential are given by
\beq  \label{metricandhd}
e^u = |\mu|(1+|\nu|^2), \quad \quad Q = -2\mu \nu_z.
\eeq
Note that  $s$ and $r$ are both holomorphic, and so no meromorphic functions are used in this alternative representation.
\\

\subsection{The loop group frames for a minimal surface}

Comparing the extended coordinate frame $\hat F_C$ of a CMC surface with the normalized
potential $\hat \eta$, one deduces that the surface is minimal if and only if the normalized potential is of a simple form:
\bdm
\hat \eta = \bbar 0 & 0 \\ p & 0 \ebar \lambda^{-1} \dd z,
\edm
where $p$ is some meromorphic function.\\

For the coordinate frame $F_C$ derived above, the extended  frame $\hat F_C$, defined  by (\ref{withlambda}),  has the simple expression:
\beq \label{mincoordframe}
\hat F_C = e^{-u/2} \bbar  s & \lambda \bar r\\ - \lambda^{-1} r & \bar s \ebar
   =  \dfrac{1}{\sqrt{|\mu| (1+|\nu|^2)}} \bbar \sqrt{\mu} & \lambda \overline{\nu \sqrt{\mu}} \vspace{1ex}\\
      - \lambda^{-1} \nu \sqrt{\mu}  & \overline{\sqrt{\mu}} \ebar.
\eeq

On the other hand, as we showed for $\hat \eta_0$ in the proof of Theorem \ref{thm1}, we can, on the set $\Sigma^\circ$ on which $p$ has no poles, integrate $\hat \eta$ explicitly, with the initial condition $\hat \Phi(z_0) = I$,
and perform an 
Iwasawa decomposition to obtain another frame 
\bdm 
\hat F =  \hat E_0 \frac{1}{\sqrt{1+|q|^2}} \bbar 1 & -\lambda \bar q\\ \lambda^{-1} q & 1 \ebar,
\quad \quad 
\hat E_0 = \hat F_C(z_0) = \bbar A_0 &  B_0 \lambda \\ -\bar B_0 \lambda^{-1} & \bar A_0 \ebar,
\edm
where $q(z) = \int_{z_0}^z p(\tau) \dd \tau$. The integral is well defined because
any poles of $p$, on $\Sigma$, are assumed to be of even order. 
This amounts to:
\beq  \label{minnormframe}
\hat F =\frac{1}{\sqrt{1+|q|^2}} \bbar A_0 + B_0 q & \lambda(-A_0 \bar q + B_0) \\
  \lambda^{-1}(\bar A_0 q - \bar B_0) & \bar A_0 + \bar B_0 \bar q \ebar.
\eeq

 The frame $\hat F$
 differs from $\hat F_C$ by right multiplication by a map into $\SSS^1$. 
In other words
\bdm
\hat F_C = \hat F \bbar e^{i \theta} & 0 \\ 0 &  e^{-i \theta}\ebar.
\edm

As pointed out previously (at (\ref{Czeroexpression})), we have
 $\sym_h(\hat F) = 0$ for the type of loop given at (\ref{minnormframe}), so a minimal surface
is not likely to be obtained from its extended frame by a variant of the
 Sym-Bobenko formula. 
  This is consistent with the fact that, unlike a non-minimal CMC surface, a minimal surface is not determined by its Gauss map -- one needs to know the  Hopf differential as well.  \\

Comparing $\hat F_C$ and $\hat F$, we see that the ratio $\nu/(-1)$ is the same 
as $(\bar A_0 q - \bar B_0)/( A_0  + B_0 q)$ which is solved to get 
\beq   \label{qformula}
q = \frac{\bar B_0  - A_0 \nu}{\bar A_0  + B_0 \nu},
\eeq
 and hence a formula for $p = q_z$ in terms of $\mu$ and $\nu$.
We summarize this as:

\begin{proposition}   \label{minimalprop}
Let $\Sigma \subset \C$ be a simply connected domain,
and $f: \Sigma \to \real^3$  a
minimal immersion with classical Weierstrass data $\mu$ and $\nu$.
Choose any  basepoint $z_0 \in \Sigma$. 
Let $\hat F_C$ be the extended coordinate frame given by (\ref{mincoordframe}), 
and (\ref{framedef}), denoting the initial data for $F_C(z_0)$ by
\beq   \label{initialdata}
 E_0 = \bbar A_0 & B_0 \\ - \bar B_0 & \bar A_0 \ebar,
 \quad A_0 = \frac{\sqrt{\mu_0}}{\sqrt{|\mu_0| (|\nu_0|^2+1)}},
 \quad B_0 =  \frac{\overline{\nu_0 \sqrt{\mu_0}}}{ \sqrt{|\mu_0|(|\nu_0|^2+1)}}.
\eeq
Then the Hopf differential and normalized potential for $\hat F$ are given by
\beq  \label{Qandp}
Q = -2\mu \nu_z, \quad \hat \eta = \bbar 0 & 0 \\ p \lambda^{-1} & 0 \ebar \dd z, \quad
p = -\frac{\nu_z}{(\bar A_0  + B_0 \nu)^2}.  \vspace{2ex}
\eeq

Conversely:  
Let $Q$ be a holomorphic function and $p$ a meromorphic function on $\Sigma$,
such that:
\begin{enumerate}
\item  \label{ass1}
$\Ord (p) \neq -1$;
\item   \label{ass2}
$Q/p$ is holomorphic, and we have $\Ord(Q)(z) \geq -\Ord(p)(z) -2$
at any pole $z$ of $p$.
\end{enumerate}   
Let $z_0 \in \Sigma \setminus \{\textup{poles of $p$}\} \cup \{\textup{zeros of }Q/p\}$.  Set  
\bdm
 q = \int_{z_0}^z p(\tau) \dd \tau,
 \quad  \quad  \bar A_0 := \left(\frac{ Q(z_0)}{p(z_0)}\frac{\bar  p(z_0)}{\bar  Q(z_0)}\right)^{1/4}
\edm
and
\bdm
 \quad \nu :=  -\bar A_0^2 q, \quad
\quad \mu := \dfrac{Q}{2 p} \bar A_0^{-2}.
\edm
Then $\mu$ and $\nu$ are the Weierstrass data for a unique 
 minimal surface (possibly with branch points) $f: \Sigma \to \real^3$, given by
the Weierstrass formulae (\ref{wformula}) and (\ref{classicalw}).  The surface is regular
at any point where either
\begin{enumerate}
\item [(i)]   $\Ord(Q)=\Ord(p)$, or
\item [(ii)]   $\Ord(Q) = -\Ord(p)-2$, where $\Ord(p) \leq -2$. 
\end{enumerate}
The Hopf differential is given by $Q\dd z^2$, 
 the coordinate frame for $f$ has
initial condition
\beq  \label{wdataic}
F(z_0) = \bbar A_0 & 0 \\ 0 & \bar A_0 \ebar,
\eeq
and the normalized potential for $\hat F$ is given by 
\bdm
\bbar 0 & 0 \\ p \lambda^{-1} & 0 \ebar \dd z.
\edm.
\end{proposition}

\begin{proof}
The formulae at (\ref{Qandp}) for $Q$ and $p$ in terms of $\mu$ and $\nu$ have
already been derived, since $p$ is obtained by differentiating the formula
(\ref{qformula}) for $q$. \\

For the converse, first note that $q$ is well defined because of the assumption
(\ref{ass1}) on the poles of $p$. Next, $\nu$ is meromorphic, $\mu$ is holomorphic 
and the assumption (\ref{ass2}) that $Q/p$ is holomorphic implies that, if
$p$ has a zero at a point then
$\Ord(\mu \nu^2) = \Ord(-Qq^2/(2p))= \Ord(Q) + \Ord(p) + 2>0$ at this point. On the other 
hand if $p$ has a pole then the assumption on the orders of vanishing 
implies that $\Ord(\mu \nu^2) \geq 0$.
Thus $\mu \nu^2$ is holomorphic.  Hence $\mu$ and $\nu$ can be taken as 
the Weierstrass
data for a minimal surface $f: \Sigma \to \real^3$. The surface is regular at points
where  $\Ord(\mu) = 0$ and $\Ord(\nu) \geq 0$,  or $\Ord(\mu) = -2 \Ord(\nu) \geq 0$, which
translates to the conditions (i) and (ii).\\

Finally, we have $\nu_0 = \nu(z_0) = 0$, and $\mu_0 = \mu(z_0)=Q(z_0)/(2 \bar A^2_0 p(z_0))$,
and the   initial condition $E_0$ from (\ref{initialdata}) is that  stated at (\ref{wdataic}).
By the first part of the Theorem, the corresponding normalized potential and  the Hopf differential for $f$
 are given by the formula at (\ref{Qandp}), 
as 

\bdm
\hat \eta = \bbar 0 & 0 \\ \tilde p \lambda^{-1} & 0 \ebar \dd z, \quad
\tilde p = -\frac{\nu_z}{(\bar A_0 + B_0 \nu )^2} = -\frac{-\bar A_0^2 p}{\bar A_0^2} = p,
\edm
and $\tilde Q = -2 \mu \nu_z = (Q/(p))\bar A_0^{-2} (q_z \bar A_0^2) = Q$.\\
\end{proof}

\subsection{Proof of Theorem \ref{mainthm}}  \label{thmproofsection}
Consider the meromorphic data $(\hat \eta_0,  Q)$, where
\bdm
\hat \eta_0 = \bbar 0 & 0 \\ p  & 0 \ebar \lambda^{-1} \dd z,
\edm
associated, by Proposition \ref{minimalprop}, to a minimal immersion $f: \Sigma \to \E^3$. The data is unique given a basepoint $z_0$.   After a translation and rotation of $\E^3$, and a simple change of conformal coordinates, we 
may assume that
\bdm
\mu(z_0)= 1, \quad \quad \nu(z_0) = 0, 
\edm
so that $\hat E_0 = I$. 
The function $p$ then simplifies to $-\nu_z$ and the function
$a=Q/p$ is then, by (\ref{Qandp}) reduced to
\bdm
a: = Q/p = 2\mu.
\edm
This is holomorphic, and so we can define a normalized
meromorphic potential
\beq   \label{simpleetah}
\hat \eta_h = \bbar 0 & -\frac{hQ}{2p} \\ p & 0 \ebar \lambda^{-1} \dd z
= \bbar 0 & -h \mu  \\ - \nu_z & 0 \ebar \lambda^{-1} \dd z,
\eeq
for any real value of $h$. \\

According to Theorem \ref{thm1}, there is a continuous family of  CMC $h$
 surfaces $f_h$ associated to $\hat \eta_h$, which includes a minimal surface $f_0$. 
We now show that  $f_0$ is  the original surface $f$:

\begin{theorem}   \label{thm2}
Let $f$, $\hat \eta_h$ and $z_0$ be as above.  Then, for every $h \neq 0$, 
the $1$-form $\hat \eta_h$ is the
normalized potential for a unique immersed CMC $h$ surface $f_h: \Sigma \to \E^3$,
obtained via the DPW construction with initial condition
 $\hat E_0 = I$. 
 The map $f_0: \Sigma \to \E^3$, obtained from Theorem \ref{thm1} as $f_0(z) = \mathcal{F}(z,0)$,
is identical with $f$.
\end{theorem}

\begin{proof}
 To see that the CMC-$h$ surface corresponding
to $\hat \eta_h$ is smooth we can use the conditions in Theorem \ref{zeropoletheorem}:  here $a = 2\mu$ and $Q = -2\mu \nu_z$.
By the assumptions on $\mu$ and $\nu$, the function $a$ has no poles,
and  $a$ has a zero if and only if it is  of order $n=2k$ and the function 
$\nu$ has a zero of order $k$.  In this case, we have
 $\Ord(Q)=k-1 = \frac{n+2}{2} - 2$, which satisfies the second condition 
of Theorem \ref{zeropoletheorem}, with $r=1$. \\

Finally, we must show that $f_0 = f$.  We know from Theorem \ref{thm1} that $f_0\big|_{\Sigma^\circ}$ is an immersed minimal surface, and from
 equation (\ref{f0metric}) the metric is given by 
$e^u = \frac{1}{2} |f_x| = \frac{1}{2}(1+|g|^2)|a|$. 
In the present situation, $a=2\mu$ and $g=q=-\nu$, and
so we have 
\bdm
e^u = |\mu|(1+|\nu|^2).
\edm
This is the same as the formula at (\ref{metricandhd}) for 
the metric of $f$, and so $f$ and $f_0$ have the same metric on 
$\Sigma^\circ$. They also have the same Gauss map and Hopf differential.  Hence
they are the same surface up to an isometry. Finally, it follows from the choice of 
initial condition $F(z_0)$ for both surfaces, that both maps satisfy
$f(z_0)=0$ and $f_z(z_0) = e_1 - ie_2$.  Hence $f\big|_{\Sigma^\circ}=f\big|_{\Sigma^\circ}$.
Since both maps are real analytic, and $\Sigma^\circ$ is dense, they are identical. \\
\end{proof}

We can now finish the proof of Theorem \ref{mainthm}:
\begin{proof} 
The first item is Theorem \ref{thm2}.  The second item follows from 
Theorem \ref{thm1} for $h \neq 0$ and, for $h=0$, from the converse part of Proposition \ref{minimalprop}:   suppose that
conditions (a) and (b) of Theorem \ref{mainthm} are satisfied. 
  Since $a(z_0)$ is assumed real, the constant 
  $A_0$ is equal to $1$. 
   We first need to check that conditions (1) and (2) of the proposition are satisfied.
   Here $p=Q/a$, and it follows from the conditions on the vanishing
orders of $Q$ and $a$ in Theorem \ref{zeropoletheorem} that this cannot have a pole of
order $1$, giving condition (1).  Condition (\ref{ass2}) 
  is equivalent to our assumptions \ref{conda} and \ref{condb} on the orders
 of vanishing of $Q$ and $a$.  
Thus, $f_0$ is a, possibly branched, minimal immersion, with the given
Weierstrass data $\mu$ and $\nu$. 
Finally, the regularity conditions (i)  and (ii) of 
Proposition \ref{minimalprop} are equivalent to condition (b) of the theorem,
and so $f_0: \Sigma^* \to \E^3$ is immersed. \\
 \end{proof}

\begin{remark}
If the classical Weierstrass data $\mu$ and $\nu$ are not given such that 
$\mu(z_0) =1$ and $\nu(z_0)=0$:   using the  general
formulae for $Q$ and $p$ given at (\ref{Qandp}), one obtains in the place of
(\ref{simpleetah}) the formula:
\beq \label{generaleta}
\hat \eta_h = \bbar 0 & -h \mu \Gamma_0 (\bar \nu_0 \nu +1)^2 \vspace{1ex} \\
             -\dfrac{\nu_z}{\Gamma_0 (\bar \nu_0 \nu +1)^2} & 0 \ebar \lambda^{-1} \dd z, \quad
            \Gamma _0:= 
          \dfrac{\bar \mu_0 }{|\mu_0| (|\nu_0|^2+1)},
\eeq
where $\nu_0 = \nu(z_0)$ and $\mu_0=(\mu(z_0))$.\\
\end{remark}
   
\begin{remark}   \label{pcremark}
Considering the expressions at (\ref{fundamentalforms}) for the first and second fundamental forms, the fact that the Hopf differential is the same for every $h$
shows that not only umbilics, but also principal curves, are  preserved under the deformation.  The Weingarten
matrix for $f_H$ is
\bdm
W_H = I^{-1} II = \bbar H + \frac{1}{2}e^{-2u} \, \Re Q   & -\frac{1}{2}e^{-2u} \, \Im Q \\
              -\frac{1}{2}e^{-2u} \,\Im Q & H - \frac{1}{2}e^{-2u} \,\Re Q \ebar.
\edm
At an umbilic point there are no principal directions. In a neighbourhood of a point which is not umbilic we can always assume that coordinates are chosen so that $Q$ is real; that is, the coordinates are isothermic (conformal principal coordinates) and
the Weingarten matrix is diagonal. Since $Q$ is constant with respect to $h$ under the deformation, these coordinates are also isothermic for $f_h$,
for every $h$.  \\
\end{remark}

\begin{remark}  \label{pcremark2}
The expression for the Weingarten matrix $W_h$ also shows what happens geometrically around the basepoint $z_0$ as $h$ varies.  Assuming $Q$ is real, the principal curvatures are the eigenvalues:
$\kappa_\pm = h \pm \frac{1}{2}e^{-2u} Q$.  
Comparing the coordinate frame $\hat F_{C,h}$ with the  meromorphic frame
$\hat \Phi_h$, and using their relationship via the Iwasawa decomposition $\hat \Phi_h = \hat F \hat B_+$, where $\hat B_+(z_0) = I$, we have that 
$e^{u(z_0)} = a(z_0)/2$, independent of $h$.  Thus the principal curvatures for $f_h$ at the basepoint  are given by
\bdm
\kappa _\pm (z_0) =  \pm C_0 + h, \quad \quad C_0 := \frac{1}{2}e^{-2u(z_0)} Q(z_0).
\edm
Assuming the basepoint is not umbilic so that we can take $C_0>0$,
we see that the surface is negatively curved at the basepoint for $h \in (-C_0, C_0)$, flat at $z_0$ for $h=\pm C_0$ and positively curved for $|h|>C_0$. As
$h$ grows large, the surface becomes spherical around the point $z_0$.

\end{remark}

\section{Examples}
\subsection{Non-minimal CMC surfaces associated to well-known minimal surfaces}
As mentioned in the introduction, Theorem \ref{mainthm} gives a means to define non-minimal CMC surfaces from known minimal surfaces. In general one has the question of
which basepoint to choose, as different choices will result in different non-minimal surfaces.  We consider here some examples where there is a canonical choice of basepoint: we will show in the next section that if a CMC  surface $f_H$ has a reflective symmetry with respect to a plane in $\real^3$ then, if the basepoint is chosen to be some point on the intersection of this plane with the surface, the associated surfaces $f_h$, for $h\in\real$ also have the same symmetry.  A similar statement holds for finite order rotational symmetries about some axis.   \\

In the examples below, the basepoints are chosen for the following reasons: for the sphere, all points are the same.  For the catenoid, the circle of smallest radius lies in a plane of symmetry of the surface - choosing a basepoint on this circle will result in non-minimal CMC surfaces with the same planar symmetry, and all points on the circle are the same geometrically.  For the helicoid, any point on the central axis is geometrically the same, and such a point is a natural choice.   Enneper's surface of order $k$ has a finite order rotational symmetry about an axis through the point $z_0=0$, and so this is a natural choice of basepoint to produce CMC surfaces with the same symmetry.

\begin{example}\label{planeexample}
The simplest minimal surface is the plane, which has Weierstrass data $\mu =\mu_0$,
$\nu=\nu_0$, where $\mu_0$ and $\nu_0$ are non-zero constants.
Using the formula (\ref{generaleta}), we obtain the potential
\bdm
\hat \eta = \bbar 0 & -h |\mu_0|(1+|\nu_0|^2) \\ 0 & 0 \ebar \lambda^{-1} \dd z.
\edm
This is the potential for a once-punctured round sphere of radius $1/h$. Note that in this
example $f_0$ is complete, but $f_h$ is not complete for $h \neq 0$.  More precisely, $f_0$ has a
planar end as $|z| \to \infty$, whilst $f_h$ has a finite limit.\\
\end{example}


\begin{figure}[ht]
\centering
$
\begin{array}{ccc}
\includegraphics[height=25mm]{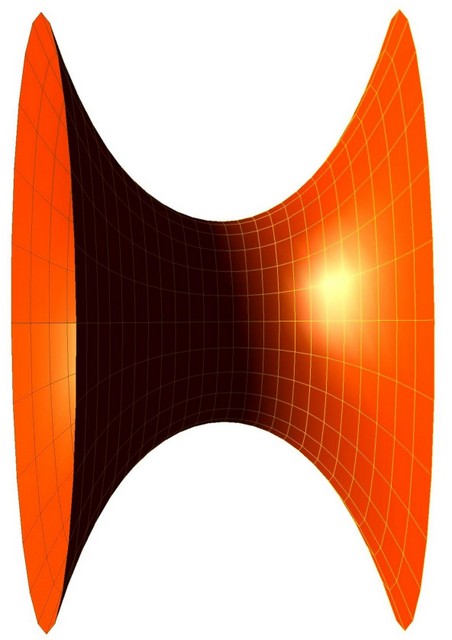} \quad &  \quad
\includegraphics[height=25mm]{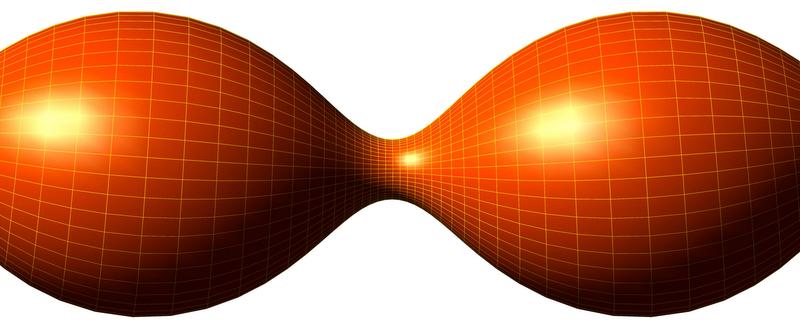} \quad &  \quad
 \includegraphics[height=25mm]{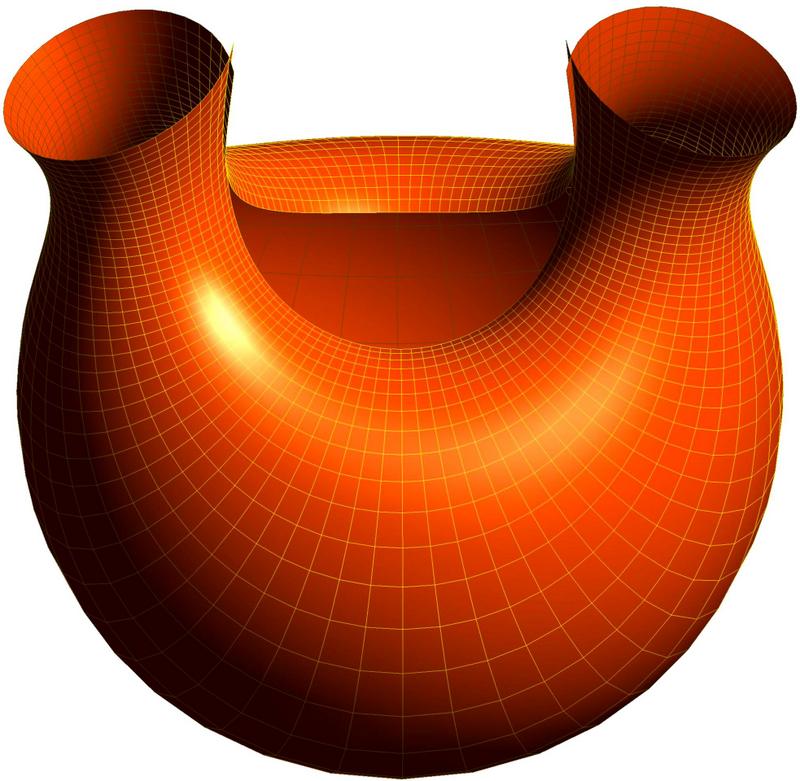} \\
 H=10^{-10} & 
 H=0.1 & H=10   \vspace{1ex}\\
\end{array}
$
\caption{Various CMC $H$ surfaces in  the catenoid family, with basepoint on the catenoid chosen on the parallel of smallest radius.} 
\label{figurecat}
\end{figure}



\begin{example}\label{catexample}
\emph{The Catenoid:} taking the Weierstrass data $\mu = -e^{-z}/2$ and $\nu = -e^z$ on $\C^2$ gives a covering of the catenoid.   We compute the potentials for the associated non-minimal CMC surfaces.  The function $\nu$ is never zero, so we use the formula
(\ref{generaleta}) with basepoint $z_0=0$.   Here $\nu_0 = -1$, $\mu_0 = -1/2$ and the formula
(\ref{generaleta}) gives
\bdm
\hat \eta = \bbar 0 & -h\frac{1}{4} e^{-z}(e^z+1)^2 \\ -2\, e^z(e^z+1)^{-2} & 0 \ebar \lambda^{-1} \dd z,  \quad \quad z_0 = 0.\\
\edm
Some examples are plotted in Figure \ref{figurecat}.  The surface with $H=0.1$ looks rather like an unduloid, but it does not close up.  The surface, around the basepoint $z_0$, shrinks to a tiny sphere as $H$ grows large. 
\end{example}

\begin{example}\label{helexample}
\emph{The Helicoid:} The helicoid is obtained by multiplying $\mu$ by $i$ in the data for
the catenoid.  Thus the potentials for the associated surfaces are
\bdm
\hat \eta = \bbar 0 & \frac{ih}{4}e^{-z}(e^z+1)^2 \\ 2e^z(e^z+1)^{-2} & 0 \ebar \lambda^{-1} \dd z,  \quad \quad z_0 = 0. \\
\edm
Plots for different values of $H$ are shown in Figure \ref{figurehel}.
Three plots of the "almost minimal" surface with $H=0.001$ are shown in
 Figure \ref{figurehel2}.
  On the large scale it looks like a chain of spheres, although it does not close up. The  "sphere" shown is of radius $1000$.  The third image is at the center of the second image at close range.

\end{example}


\begin{figure}[ht]
\centering
$
\begin{array}{ccc}
\includegraphics[height=25mm]{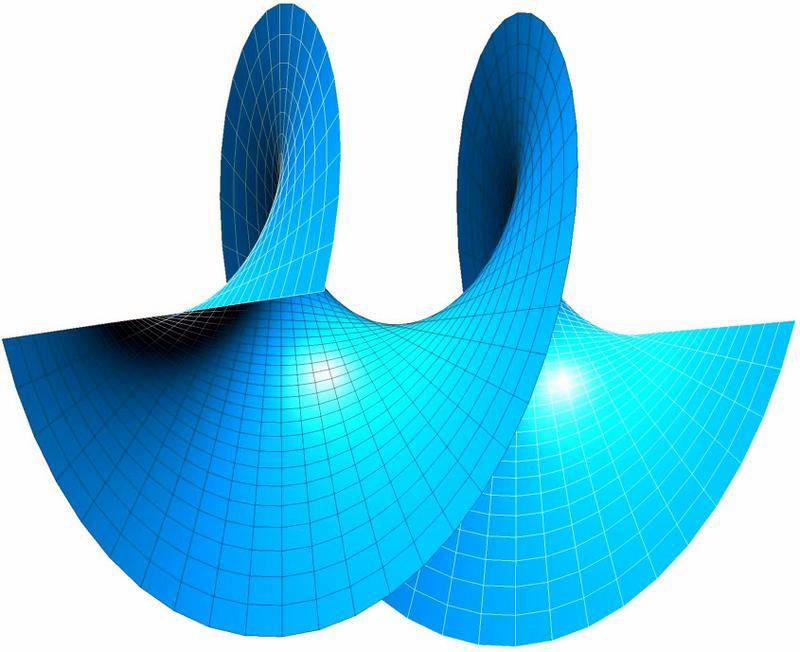} \quad & \quad
\includegraphics[height=25mm]{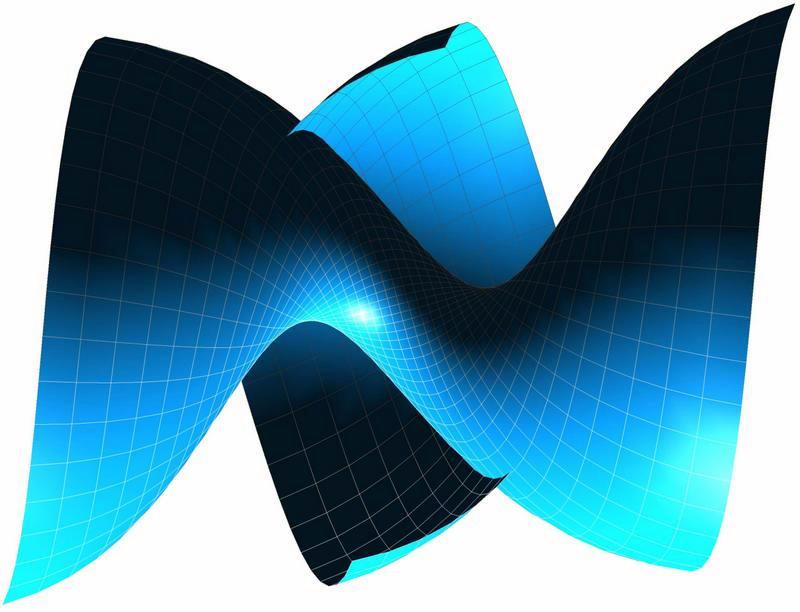}  \quad & \quad
\includegraphics[height=25mm]{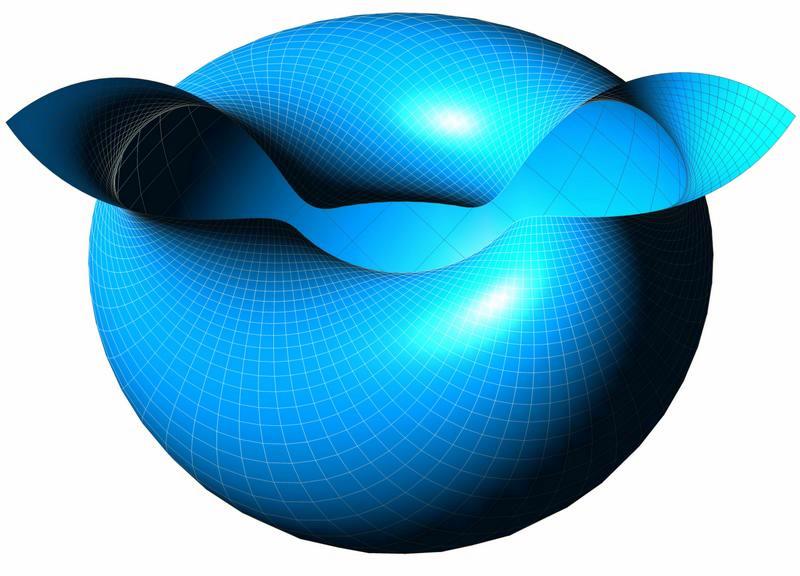}   

 \vspace{2ex} \\
 H=10^{-10}, & 
 H=0.1   &    H=5
\end{array}
$
\caption{CMC $H$ surfaces in the helicoid family with basepoint chosen on the central axis.} 
\label{figurehel}
\end{figure}

\begin{figure}[ht]
\centering$
\begin{array}{ccc}
\includegraphics[height=30mm]{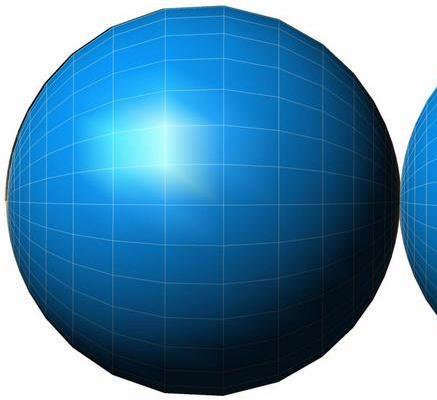} \quad & \quad
 \includegraphics[height=30mm]{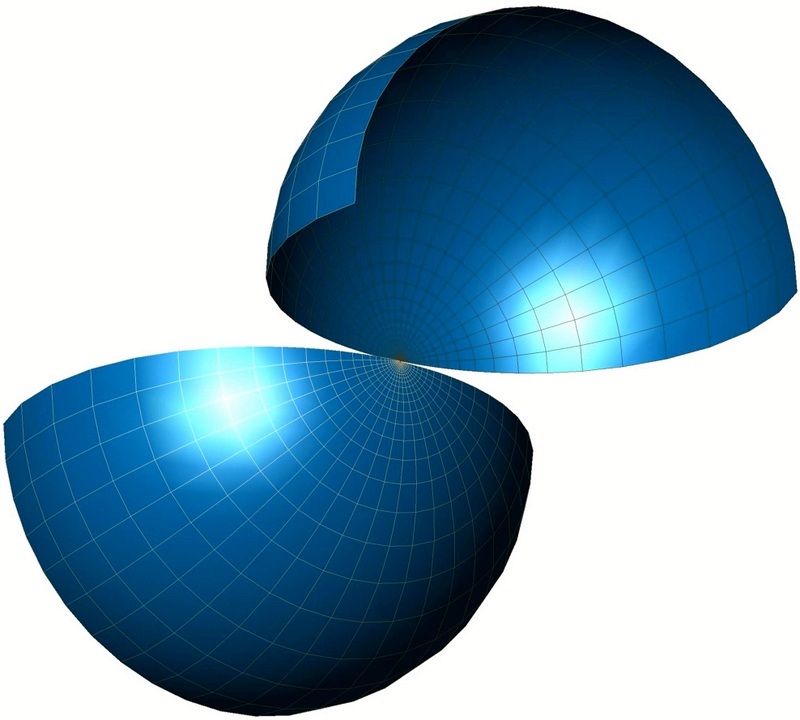} \quad & \quad
\includegraphics[height=30mm]{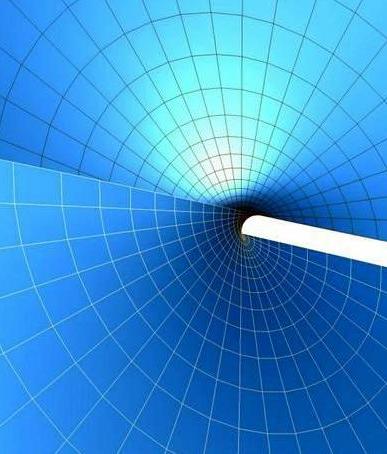} 
 \end{array}
$
\caption{Several plots of the helicoidal surface of CMC $H=0.001$.   } 
\label{figurehel2}
\end{figure}


\begin{example} \label{enneperexample}
\emph{Enneper's surface:} Enneper's surface of order $k \geq 1\emph{}$ is given by 
 $\mu = 1$, $\nu = z^k$
on $\C$.  This gives associated CMC-$h$ surfaces with potentials
\bdm
\hat \eta = \bbar 0 & h \\ kz^{k-1} & 0 \ebar \lambda^{-1} \dd z,  \quad \quad z_0 = 0.
\edm
These potentials are known: for the case $k=1$ and $h=1$ we obtain a round cylinder.
However the surfaces do not close up in general.
The other cases are known as  Smyth surfaces -- defined by B. Smyth in \cite{smyth}) -- or  $(k+1)$-legged Mister Bubbles. See Figure \ref{figuresmyth}.\\
\end{example}

\section{The dressing action}   \label{dressingsection}
The dressing action is an action by $\Lambda_P ^+ G^\C_\sigma$ on the space of CMC
immersions from a simply connected domain $\Sigma \subset \C$ into $\E^3$. A description of the action can be found in \cite{wu1997}.  The action is defined as follows:  for a loop $h_+ \in  \Lambda _P ^+ G^\C_\sigma$, and a CMC immersion
with extended frame $\hat F$, the pointwise Iwasawa decomposition
$$
h_+ \hat F(z) = {\widetilde F}(z) {\widetilde G}_+(z), \quad \quad
  \widetilde F(z) \in \Lambda G_\sigma, \quad \widetilde G_+(z) \in \Lambda _P^+ G^\C_\sigma,
$$
gives an extended frame $\widetilde F$ for a new CMC surface.  Note that this can be equivalently
defined by the Iwasawa decomposition  $h_+ \hat \Phi(z) = \widetilde F(z) \widetilde C_+(z)$, where $\hat \Phi$ is a meromorphic extended frame for $f$, to get the same 
extended frame $\widetilde F$.\\

It is  a question of interest whether two CMC surfaces are in the same dressing orbit: for example all CMC tori are in the dressing orbit of the cylinder
\cite{dorfwu1993}.\\

Note that if $\tilde f= (h_+)^\# f$ is obtained from $f$ by dressing by $h_+$,
and $\widetilde \Phi$ and $\hat \Phi$ are their respective normalized meromorphic frames,
with basepoint $z_0$, then writing the normalized Birkhoff decompositions
$\hat F = \hat \Phi \hat B_+$ and
$\widetilde F = \widetilde \Phi \widetilde B_+$, and substituting in the above relation
$\widetilde F = h_+ \hat F \widetilde B_+^{-1}$, we obtain the relation
\bdm
\widetilde \Phi = h_+ \hat \Phi \hat W_+, \quad \quad 
 \hat W_+(z) \in \Lambda^+ _P G^\C_\sigma, \quad \hat W_+(z_0) = h_+^{-1}.
\edm
where $\hat W_+ = \hat B_+ \widetilde G_+^{-1} \widetilde B_+^{-1}$.
If we denote by $\Sigma^*$ the open dense set on which both $\widetilde \Phi$ and
$\hat \Phi$ have no poles, then the above formula shows that $\hat W_+: \Sigma^* \to \Lambda^+ G^\C_\sigma$ is holomorphic.  \\

Conversely, given $f$ and $\hat \Phi$ as above, any holomorphic map $\hat W_+$
from a neighbourhood $U$ of $z_0$ into $\Lambda^+ _P G^\C_\sigma$  
corresponds to a dressing by the element $h_+ = \hat W_+(z_0)^{-1}$.
The new solution has meromorphic frame 
$\widetilde \Phi = \hat W_+(z_0)^{-1} \hat \Phi \hat W_+$.
At the level of potentials, the new surface has the potential
\beq  \label{dressedpotential}
 \hat \eta \# \hat W_+ := \hat W_+^{-1} \hat \eta \hat W_+ + \hat W_+^{-1} \dd \hat W_+.
\eeq
Thus if $\hat \eta_H = \bbar 0 & -\frac{H}{2}a \\ \frac{Q}{a} & 0 \ebar \lambda^{-1} \dd z$, we have
\bdm
\widetilde \eta_H = \hat \eta \# \hat W_+
= \bbar 0 & -\frac{H}{2} a \rho^2 \\ \frac{Q}{a \rho^2} & 0 \ebar \lambda^{-1} \dd z, \quad \quad 
\hat W_+ = \bbar \rho^{-1} & 0 \\ 0 & \rho \ebar + o(\lambda).
\edm
In other words the data $(a(z), Q(z))$ are dressed to 
$(\rho^2(z) a(z), Q(z))$, where $\rho$ is some meromorphic function. In particular the Hopf differential is the same for both surfaces.\\

We say that $\hat \eta$ and $\widetilde \eta$ are \emph{dressing equivalent}
if (\ref{dressedpotential}) holds for some meromorphic function $\hat W_+$,
and \emph{formally} dressing equivalent if the relation holds for 
some formal power series $\hat W_+$. \\

If we take the potentials $\hat \eta_H$ and $\widetilde \eta_H$ above, and
let $H$ vary, then, for $\hat \eta_h$ and $\widetilde \eta_h$ to be dressing equivalent, the gauge $\hat W_+$ will also depend on $h$.  Thus, even if
$\hat \eta_H$ is dressing equivalent to $\widetilde \eta_H$ at some value $H$ 
of
$h$, it is not automatic
that the whole families of surfaces are dressing equivalent.

\begin{theorem}   \label{dressingthm1}
 Let  $\hat \eta_h$ and $\widetilde \eta_h$ be two normalized meromorphic potentials
 with meromorphic data $(a, Q)$ and $(\tilde a, Q)$, and let $z_0$ be the associated
 basepoint.  Suppose that either $Q(z_0) \neq 0$ or  $z_0$ is a simple root of $Q$. Then:
 \begin{enumerate}
 \item  \label{dressingthmitem1}
The potentials  $\hat \eta_h$ and $\widetilde \eta_h$ are formally dressing equivalent for
 every $h\neq 0$.  The gauge $\hat W_{+,h}$ can be computed locally around $z_0$.
\item   \label{dressingthmitem2}
The map $\hat W_{+,h}$ extends to $h=0$ if and only if it is constant in $h$ and
 has the form
\beqa
\hat W_+= \bbar a_0 & b_1 \lambda \\ 0 & a_0^{-1} \ebar,
      \label{Wplusformula} \\
a_0 = \sqrt{\frac{a}{\tilde a}}, \quad \quad  b_1 = \frac{\tilde a}{Q} \frac{\dd }{\dd z}\sqrt{\frac{a}{\tilde a}}, \quad \textup{where } \frac{\dd b_1}{\dd z}=0.  \label{a0b1formulae}
\eeqa
\end{enumerate}
\end{theorem}
\begin{proof}
  Wu \cite{wu1997} has studied the general problem of dressing two normalized potentials $\hat \eta$ and $\widetilde \eta$ into each other.  Formally, it is
enough to solve the equations (5.11)-(5.16) in \cite{wu1997} for the map $\hat W_+$.
(The symbols $E$, $p$ and $q$ used in \cite{wu1997} corresponds to ours via $E = -\frac{h}{2} Q$,
$p = -\frac{h}{2} a$, and $q= -\frac{h}{2} \tilde a$.)
 Writing 
$$\hat W_{+,h}(z) = \bbar \sum_{k=0}^\infty  \,  a_{2k}(z) \,  \lambda^{2k}  &
         \sum_{k=0}^\infty  \,  b_{2k+1}(z) \,  \lambda^{2k+1} \\
          \sum_{k=0}^\infty   \, c_{2k+1}(z) \,  \lambda^{2k+1}  &
          \sum_{k=0}^\infty  \,  d_{2k}(z) \,  \lambda^{2k}  \ebar,
$$
 the equations (also correcting two typographic errors  in \cite{wu1997})  to be solved are:
\beqa
a_0  = d_0^{-1}  =\sqrt{\frac{a}{\tilde a}},  \label{dressing1} \\
2b^\prime_n + b_n \left(Q^\prime - \left(\frac{a^\prime}{a} + \frac{\tilde a^\prime}{\tilde a}\right) Q\right) = \left( a_{n-1}^{\prime \prime} - a_{n-1}^\prime \frac{a^\prime}{a}\right) \tilde a, \quad \quad (n\geq 1),
         \label{dressing2} \\
c_n = \frac{2}{h}\left( -\frac{b_n Q}{ a \tilde a} + \frac{ a_{n-1}^\prime}{ a} \right), 
      \quad \quad (n \geq 1),
           \label{dressing3}\\     
 a_2 = \frac{1}{2}a_0 b_1 c_1 - \frac{b_1^\prime}{h \tilde a}, 
    \quad \quad
  d_2 = \frac{1}{2}d_0 b_1 c_1 + \frac{b_1^\prime}{h a},           
           \label{dressing4}\\
\quad a_n= \frac{1}{2}a_0 \sum_{j=0}^{n/2-1} b_{2j+1} c_{n-2j-1} 
       - \frac{1}{2}a_0 \sum_{j=1}^{n/2-1} a_{2j}d_{n-2j} - \frac{b^\prime_{n-1}}{h\tilde a}, \quad (n \geq 4),
                \label{dressing5}\\
\quad d_n= \frac{1}{2}d_0 \sum_{j=0}^{n/2-1} b_{2j+1} c_{n-2j-1} 
       - \frac{1}{2}d_0 \sum_{j=1}^{n/2-1} a_{2j}d_{n-2j} - \frac{b^\prime_{n-1}}{h a}, \quad (n \geq 4).  \label{dressing6}  
\eeqa
As discussed in \cite{wu1997}, Theorem 5.17, there is always a solution $\hat W_{+,h}$ 
for any $h \neq 0$, which takes care of item \ref{dressingthmitem1} of this theorem. \\

We now consider item \ref{dressingthmitem2}.   Assume  that the solution extends to $h=0$.
We argue by induction on $n$ that: 
\beqas
a_n \textup{ is independent of $h$ for all $n$ },
\quad a_n =d_n=0 \quad \textup{for }  n>0,\\
c_n=0 \quad \textup{for all } n, \quad \quad
b_n=0 \quad \textup{for } n>1.
\eeqas

 Clearly, our hypothesis holds for $n=0$, by equation (\ref{dressing1}).
Assume now that the hypothesis holds for $n-1$, and consider $n$. If $n$ is odd,
we need to consider only $b_{n}$ and $c_n$.  By equation (\ref{dressing2}),
$b_{n}$ is independent of $h$, because the right hand side of the equation has this property (the inductive hypothesis).  But then the expression
$-\frac{2b_n Q}{ a \tilde a} + \frac{2 a_{n-1}^\prime}{ a}$
 on the  right hand side of 
equation (\ref{dressing3}) is also independent of $h$, and therefore, 
if $c_n$ is defined for $h=0$, the equation (\ref{dressing3}) becomes
\beq \label{bnformula}
c_n=0 \quad \textup{and} \quad  
b_n = \frac{a_{n-1}^\prime \tilde a}{Q}.
\eeq
If $n=1$ we obtain the formula at (\ref{a0b1formulae}) for $b_1$, and
if $n\geq 3$ we obtain, using the inductive hypothesis on $a_{n-1}$,
that $b_n=0$. 
That deals with the odd case.  \\

Now if $n$ is even, we are considering
$a_{n}$ and $d_{n}$.  We first discuss $a_n$. If $n=2$ we have,
from (\ref{dressing4}), that
 $a_2 = \frac{1}{2}a_0 b_1 c_1 - \frac{b_1^\prime}{h \tilde a}$.
We know that $c_1=0$, and that $b_1^\prime/\tilde a$ is independent of
$h$. Hence we must have
\beq \label{b1primeformula}
b_1^\prime= a_2=0.
\eeq
If $n\geq 4$, we use equation (\ref{dressing5}), which reduces,
by the inductive hypothesis to
$$
a_n= 0.
$$
The argument for $d_n$ is identical, using (\ref{dressing6}).

Thus we have proved the formula (\ref{Wplusformula}) for 
$\hat W_+$ by induction.
The formulae at (\ref{a0b1formulae}) are given at
(\ref{dressing1}), (\ref{bnformula}) and (\ref{b1primeformula}). 
This proves the "only if" direction of item \ref{dressingthmitem2},
and the "if" direction is just the observation that 
the stated conditions for $\hat W_+$ at (\ref{a0b1formulae}) do give a
solution to Wu's equations (\ref{dressing1})-(\ref{dressing6}).
\end{proof}

\begin{remark} 
Note that, again using Theorem 5.17 of \cite{wu1997}, an analogous
statement to Theorem \ref{dressingthm1} holds for the case that
$Q$ has an umbilic of order greater or equal to $2$.  The only difference
is that, in addition to the conditions at (\ref{a0b1formulae}),
one also requires all the dressing invariants (of which there are many
for higher order umbilics) must be satisfied by both potentials.
The dressing invariants are independent of $h$ (see Lemma 3.2 of
\cite{wu1997}), and so there is no further argument needed.
\end{remark}

\subsection{Dressing minimal surfaces}
The dressing action on minimal surfaces, defined in \cite{dorfmeisterpedittoda}
in the same way as for non-minimal surfaces, via the meromorphic potential, is
an action on equivalence classes of minimal surfaces, all of which have the
same normalized potential.  For the minimal case, there is not a unique minimal
surface related to a normalized potential. Theorem  \ref{dressingthm1} defines  a natural class of dressing elements $h_+ \in \Lambda_P ^+ G^\C_\sigma$, namely given by 
\beq   \label{minimaldressingelement}
h_+ = \hat W_+(z_0)^{-1} = \bbar a_0(z_0)^{-1} & -b_1(z_0)\lambda \\
        0 & a_0(z_0) \ebar,
\eeq
where $\hat W_+$ is as in the theorem, which can be used to 
define dressing as an action on the class of minimal surfaces.  But, 
 in \cite{dorfmeisterpedittoda}, it is also shown (in Lemma 5.1) that 
when dressing normalized potentials for minimal surfaces, one can,
without loss of generality, assume that the dressing element $h_+$
is indeed of the form (\ref{minimaldressingelement}), any other coefficients
being irrelevant.  Thus, what we have really established here is the following:
\begin{theorem} \label{dressingthm2}
Let $\Sigma \subset \C$ be a simply connected domain with basepoint $z_0$.
The dressing action, defined in \cite{dorfmeisterpedittoda} on
normalized potentials of minimal surfaces, gives a well-defined group action
on the space of minimal immersions $f: \Sigma \to \real^3$.
\end{theorem}
Note that, given an element $h_+$ of the above form,
 the dressing action on a minimal immersion can be calculated using
 the classical Weierstrass data $(\mu \dd z, \nu)$, the formulae
 in Theorem \ref{mainthm} for $a$ and $Q$, and solving the
 equations at  (\ref{a0b1formulae}) for $\tilde \mu$ and $\tilde \nu$.
 If $f$ has an umbilic of order greater than $1$ at $z_0$, one expects some complications
 in this procedure.

\section{Surfaces with symmetries}  \label{symmetrysection}
In this section we consider minimal and non-minimal CMC surfaces with symmetries that are
either reflections about a plane or finite order rotations about a point contained in the surface.  Symmetries of CMC surfaces have 
 been studied by Dorfmeister and Haak in \cite{DH98} and \cite{DH2000}, and
 a more general derivation of the potentials that correspond to symmetries can be
 found in those references. 

\begin{defn}  \label{symmetrydefn}
Let $f: \Sigma \to \real^3$ be a conformal immersion of a contractible domain.
\begin{enumerate}
\item
 We say that $f$ has a 
\emph{reflective symmetry} with respect to a plane $\Pi \subset \real^3$
if there exist conformal coordinates $z$ for $\Sigma$ such that $\Sigma$ is symmetric
about the real axis in these coordinates, that is $\overline{\Sigma}=\Sigma$, and such that: 
 $$
 f(\bar z) = R_\Pi(f(z)), \quad \quad \forall z \in \Sigma,
 $$
where $R_\Pi$ is the reflection about the plane $\Pi$.

\item
 Let $n$ be a positive integer. We say that $f$ has a \emph{fixed-point rotational symmetry of order $n$ and axis $l$} if
 there exist conformal coordinates $z$ for $\Sigma$ such that 
$e^{i\theta} \Sigma = \Sigma$, where $\theta = 2\pi/n$, and such that: 
\beq  \label{rotationdef}
f(e^{i\theta}z) = R_l f(z), \quad \quad \forall z \in \Sigma,
\eeq
 where $R_l$ is the rotation of angle $\theta$ about  the line $l$.
\end{enumerate}
\end{defn}

We will first show how these symmetries are reflected in the Weierstrass data.
For the rest of this section  we always consider a contractible domain $\Sigma \subset \C$,
with base point $z_0=0$,  and a conformal CMC $H$ immersion
$f: \Sigma \to \E^3$. If $H\neq 0$, we let $\hat \eta_H = \textup{off-diag}(-\frac{H}{2}a,  p )\lambda^{-1} \dd z$ be the associated normalized potential, where $p=Q/a$.
  If $H=0$, we let $(\mu, \nu)$ be the classical Weierstrass data,
chosen such that $\mu(0)=1$ and $\nu(0)=0$.

\subsection{Reflections about a plane}
Observe that, since the components of the matrices 
$e_1$ and $e_3$ are pure imaginary, whilst
 the components of $e_2$ are real, the reflection about the plane
$e_1 \wedge e_3$ in $\real^3 = \mathfrak{su}(2)$ is given by
\bdm
X \mapsto - \bar X.
\edm
The next lemma shows that a CMC surface is symmetric about the plane
$e_1 \wedge e_3$ if and only if coordinates and basepoint can be chosen such that 
the Weierstrass data are real along the real line. 

\begin{lemma} \label{symmetrylemma1}
Suppose  $\Sigma$ is symmetric about the real axis, that is $\overline \Sigma = \Sigma$,
and let $f:\Sigma \to \E^3$ be a CMC immersion with associated data as above.
Then 
\beq   \label{symcond1}
f(\bar z) = -\overline{f(z)}, \quad \quad  \forall \,\, z \in \Sigma,
\eeq
if and only if the associated Weierstrass data satisfy the condition
\beq   \label{symcond2}
\eta_H (z) = \overline{\eta_H(\bar z)}, \quad \quad 
\forall \,\, z \in \Sigma, \quad \quad \textup{if } H\neq 0,
\eeq
\beq \label{symcond3}
(\mu (z), \nu(z)) = \overline{(\mu(\bar z), \nu(\bar z)}, \quad \quad \textup{if } H= 0.
\eeq
\end{lemma}

\begin{proof}
First suppose that $f(\bar z) = -\overline{f(z)}$.  It follows that $f(\real \cap \Sigma) \subset \textup{span}\{e_1, e_3\}$, and so, for real values of $z$, we have $f_x$ parallel to the plane $e_1 \wedge e_3$.   It also follows from
the symmetry that the  surface is perpendicular to the plane. This means that $f_y$ must 
be parallel to $e_2$, for real $z$.
Using this fact in the definition of 
 the coordinate frame (\ref{framedef}), with initial condition $E_0 = I$,
we find that, along $\real$, the frame is $SO(2)$-valued:
\bdm
F(x,0) = \bbar \cos \theta & \sin \theta \\ - \sin \theta  & \cos \theta \ebar,
\quad \theta(x,0) \in \real.
\edm

\textbf{Case $H\neq0$:}
The above expression for $F(x,0)$ means that, for real $z$, the extended coordinate frame $\hat F$ takes values in 
$\Lambda G_{\sigma \rho}$, the fixed point subgroup with respect to the involution
$\rho$ given by $(\rho \gamma) (\lambda) := \overline{\gamma (\bar \lambda)}$, for 
$\gamma \in \Lambda G$.  The involution $\rho$ is of the first kind, which means that
the group $\Lambda G^\C_{\sigma \rho}$ is Birkhoff decomposable (see \cite{branderdorf}),
which simply means that both the factors  in
the normalized Birkhoff decomposition 
\bdm
\hat F(x,0) = \hat F_-(x,0) \hat F_+(x,0),
\edm
take values in $\Lambda G^\C_{\sigma \rho}$. Now the normalized potential is
 $\hat \eta = \hat F_-^{-1} \dd \hat F_-$, and the reality condition $\rho$,
 which is valid along the real axis,
amounts to saying that the functions $a(x,0)$ and $p(x,0)$ are real valued.
Since $\Sigma$ is connected, and the functions are meromorphic,
 this is equivalent to the condition (\ref{symcond2}). \\

Conversely, suppose that $\eta_H (z) = \overline{\eta_H(\bar z)}$ for all $z$.
It follows that
$\hat \Phi(\bar z) = \overline{\hat \Phi(z)}$ for all $z \in \Sigma$.  Let 
$\hat F$ be the unique frame obtained by the pointwise Iwasawa decomposition  
\bdm
\hat \Phi = \hat F \, \hat B_+, \quad \quad \hat F(z) \in \uu, \quad \hat B_+(z) \in \ustar. 
\edm
The corresponding decomposition for $\bar {\hat \Phi}$ is just the conjugate of this:
$\bar {\hat \Phi} = \bar {\hat F} \, \overline {\hat B_+}$.  From this, the
above symmetry of $\hat \Phi$, and the uniqueness of the Iwasawa decomposition, we conclude  that
\bdm
\hat F(\bar z) = \overline{\hat F (z)}.
\edm
Using this condition in the Sym-Bobenko formula
\bdm
 f :=  -\frac{1}{2H} \left. \left(  2 i \lambda 
\partial_\lambda \hat F \,  \hat F^{-1} \ + 
\, \hat F e_3 \hat F^{-1} - e_3 \right)\right |_{\lambda = 1},
\edm
we immediately obtain that $f(\bar z) = -\overline{f(z)}$.\\

\textbf{Case $H=0$:}  Here the coordinate frame is
\bdm
F_C = e^{-u/2} \bbar s & \bar r \\ - r & \bar s \ebar,
\edm
and it follows from the above form for $F(x,0)$ that the functions $s$ and $r$
are real valued along the real line.  Since they are holomorphic, this means
$s(z) = \overline{s(\bar z)}$ and $r(z) = \overline{r(\bar z)}$ for all $z$,
and hence $\mu = s^2$ and $\nu = sr/\mu$ have the same property.\\

Conversely, if $\mu$ and $\nu$ have the symmetry given at (\ref{symcond3}), then  the potential
$\hat \eta_h = \textup{off-diag}(-h \mu, -\nu_z)\lambda^{-1} \dd z$ has the symmetry (\ref{symcond2}).  Hence, by
the case explained above for $H\neq0$, the non-minimal CMC $h$ surface, $f_h : \Sigma \to \E^3$, associated to this potential, satisfies $f_h(\bar z) = -\overline{f_h(z)}$
for all $z$, and for every $h\neq 0$. By Theorem \ref{mainthm}, we have that this family depends continuously on $h$, and that $f_0=f$. By continuity, $f$ also
has the symmetry.

\end{proof}

\subsection{Fixed-point rotational symmetries}
In Definition \ref{symmetrydefn} for rotational symmetries, little
generality is lost by taking $l$ to be the oriented $x_3$-axis, and in our identification of $\E^3 = \mathfrak{su}(2)$ the symmetry condition
 (\ref{rotationdef}) amounts to:
\beq   \label{rotationcondition}
f(e^{i\theta}z) = \Ad_T f(z), \quad  \quad T = \bbar e^{i\theta/2} & 0 \\0 & e^{-i\theta/2} \ebar,
\quad \quad \theta = \frac{2\pi}{n}.
\eeq
This condition is reflected in the Weierstrass data for a CMC surface as follows:

\begin{lemma} \label{rotationlemma}
Suppose that the the domain $\Sigma$ is rotationally symmetric, specifically
$e^{i\theta} \Sigma = \Sigma$, where $\theta = 2\pi/n$.
Let $f:\Sigma \to \E^3$ be a conformal CMC immersion with associated data as above.
  Then
\begin{enumerate}
\item 
If $H\neq0$, then $f$ has the rotational symmetry (\ref{rotationcondition}) if and only if, for
all $z \in \Sigma$, we have
\beqas 
&&a(e^{i\theta}z) = a(z), \\
&& p(e^{i\theta}z)=e^{-2i \theta}p(z).
\eeqas

\item 
If $H=0$, then $f$ has the rotational symmetry (\ref{rotationcondition}) if and only if, for
all $z \in \Sigma$, we have
\beqas 
&&\mu(e^{i\theta}z) = \mu(z), \\
&& \nu(e^{i\theta}z)=e^{-i \theta}\nu(z).
\eeqas
\end{enumerate}
\end{lemma}
\begin{proof}  
Given the symmetry condition (\ref{rotationcondition}), set $g(z) = f(e^{i\theta}z)=\Ad_T f(z)$. Calculating the coordinate frame for $g$, 
one obtains that  the coordinate frame for $f$, 
defined by (\ref{framedef}),  satisfies 
\beq \label{framesymmetry}
F_C(e^{i\theta}z) = \Ad_T F_C(z),
\eeq
 for all $z$, and this condition is also valid for the extended coordinate frame $\hat F_C$.
   Conversely, given a CMC $H$ surface with extended frame $\hat F_C$, the Sym-Bobenko formula shows that the symmetry (\ref{framesymmetry}) implies that $f$ has the symmetry
 (\ref{rotationcondition}). This argument is also valid for the case $H =0$, taking the limit $H\to 0$ in the Sym-Bobenko formula.  Thus, in either case, the condition  
  (\ref{rotationcondition}) is equivalent with the same condition for $\hat F_C$.\\

\textbf{Case $H\neq0$:} 
The meromorphic extended frame is obtained by the normalized Birkhoff decomposition
$\hat F_C(z) = \hat \Phi(z) \hat G_+(z)$, and the condition (\ref{framesymmetry}) gives us
\beqas
\hat F_C(e^{i \theta}z) &=& \hat \Phi(e^{i \theta} z) \,  \hat G_+(e^{i \theta} z) \\
 &=& T \hat \Phi (z)  \, \hat G_+(z) T^{-1} \\
 &=& \Ad_T \hat \Phi(z) \Ad_T G_+(z).
\eeqas
The uniqueness of the $\Lambda^-_* G_\sigma^\C$ factor in the  normalized Birkhoff decomposition implies that
\bdm
\hat \Phi(e^{i \theta} z) = \Ad_T \hat \Phi(z).
\edm
  Differentiating, one obtains
\beq  \label{potentialsymmetry}
\hat \Phi^{-1} \hat \Phi_z \big|_{e^{i\theta}z} = e^{-i\theta} \Ad_T \left(\hat \Phi^{-1} \hat \Phi_z \right)\big|_z,
\eeq
which is to say 
\bdm
\bbar 0& -\frac{H}{2} \,a(e^{i\theta}z) \\ p(e^{i\theta}z) & 0 \ebar \lambda^{-1} \dd z =
\bbar 0& -\frac{H}{2} \,a(z) \\ e^{-2i \theta} \, p(z) & 0 \ebar \lambda^{-1} \dd z.
\edm
Thus $a$ and $p$ satisfy the conditions stated in the theorem.
Conversely, integrating a potential $\hat \eta$ with the symmetry given at (\ref{potentialsymmetry}),  all steps of the above argument can be reversed to conclude that the 
corresponding CMC surface $f$ has the required symmetry.\\

\textbf{Case $H=0$:} For the minimal case, the coordinate frame has the form
\bdm
F_C = e^{-u/2} \bbar s & \bar r \\ - r & \bar s \ebar,
\edm
where $\mu = s^2$ and $\mu \nu = s r$.   If the surface has the symmetry (\ref{rotationcondition}) then the metric is invariant under $z \mapsto e^{i\theta}z$, and 
hence $u(e^{i\theta}z) = u(z)$.  Thus the symmetry (\ref{framesymmetry}) of $F_C$ reduces to:  $s(e^{i\theta}z) = s(z)$ and $r(e^{i\theta}z) = e^{-i \theta}r(z)$. This implies the stated conditions for $\mu$ and $\nu$.  \\

 Conversely, given a 
minimal surface with Weierstrass data satisfying the given symmetry conditions, the associated non-minimal surface, with basepoint $z_0=0$, given in Theorem \ref{mainthm},
has normalized potential 
$\hat \eta_h = \textup{off-diag}(-h a/2, p) = \textup{off-diag}(-h \mu, -\nu_z)$. 
The symmetry assumptions on $\mu$ and $\nu$ 
give $a(e^{i \theta}z) = a(z)$ and $p(e^{i\theta}z)=e^{-2i\theta}p(z)$.
Hence this is the potential for a CMC $h$ surface satisfying the symmetry (\ref{rotationcondition}). By continuity, the symmetry also holds at $h=0$.
\end{proof}

\begin{figure}[ht]
\centering
$
\begin{array}{cc}
\includegraphics[height=40mm]{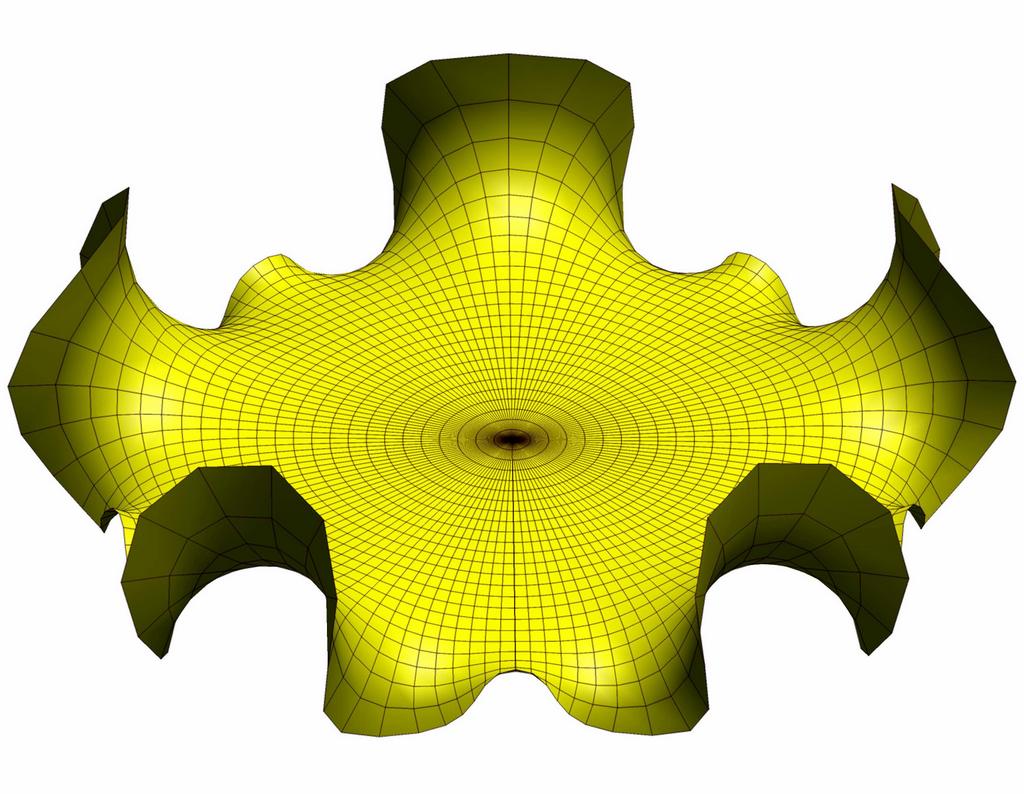} \quad & \quad
\includegraphics[height=40mm]{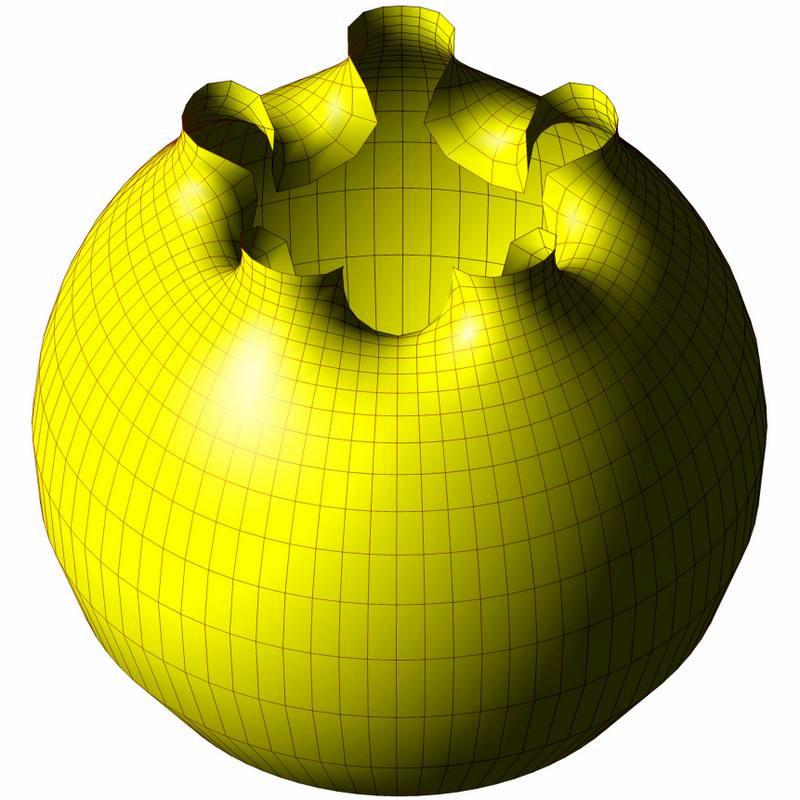}  
 \vspace{2ex} \\
 H=10^{-8}, & 
 H=2   
\end{array}
$
\caption{Two surfaces in the same family, which has an order $5$ rotational symmetry. } 
\label{figure5sym}
\end{figure}

\subsection{Preservation of symmetries under deformations}
A Euclidean motion of the ambient space will always bring a line $l$ to the $x_3$-axis, or a plane $\Pi$ to the plane $e_1\wedge e_3$. Furthermore, the symmetry conditions of the Weierstrass data are independent of $h$: in particular, for both symmetries, the Weierstrass data for a minimal surface satisfy the symmetry condition if and only if the generalized Weierstrass data for the associated non-minimal surfaces do also.  Hence the results of the
previous two lemmas have the following corollary:
\begin{theorem}  \label{symmetrythm}
Let $H$ be any real number.
Let $f_H: \Sigma \to \E^3$ be a conformal CMC $H$ immersion, and let $f_h$, for 
$h \in \real$, be the family associated by Theorem \ref{mainthm}, with basepoint $z_0=0$.
Then $f_H$ satisfies one of the symmetries at Definition \ref{symmetrydefn} if and only if $f_h$ satisfies the same symmetry for all $h$. 
\end{theorem}

\subsection{Examples with rotational symmetries}
The conditions on $a$ and $p$ in Lemma \ref{rotationlemma} are equivalent to the following
Laurent expansions for $a(z)$ and $p(z)$:
\bdm
a(z)= \sum_j a_{nj} \,  z^{nj}, \quad \quad
p(z) = \sum_j p_{nj-2}  \, z^{nj-2}.
\edm
\begin{example}
As an example with an order $5$ rotational symmetry, we numerically
 computed  solutions corresponding to the potential
\bdm
\hat \eta = \bbar 0 & -\frac{h}{2}(5.1+ 1.5 \,z^5 + 0.35 \,z^{10}) \\ 1.25 \,z^3 + 4.15 \,z^8 & 0 \ebar \lambda^{-1} \dd z.
\edm
The images of discs around the coordinate origin for the
surfaces corresponding to $h=10^{-8}$ and $h=2$ are shown at Figure \ref{figure5sym}.
\end{example}

\begin{example}
Kusner's surface with  $p=3$, defined in  \cite{kusner1987}, has Weierstrass data 
\bdm
\mu = \frac{i (\sqrt{5} z^3 + 1)^2}{(z^6+\sqrt{5}z^3-1)^2}, \quad \quad
\nu = \frac{z^2 (z^3-\sqrt{5})}{\sqrt{5}z^3+1}.
\edm
It is proved in \cite{kusner1987} that this minimal surface is complete, non-orientable, has finite total curvature $-10\pi$, has $3$ embedded flat ends, and contains $3$ straight lines which lie in a plane. The dihedral group of order 6 acts by reflections around these lines.\\

The corresponding potential, with basepoint $z_0 = 0$, is 
\bdm
\hat \eta = \bbar  0 &  -h\frac{(\sqrt{5}z^3+1)^2}{(z^6+\sqrt{5}z^3-1)^2}\\
                      -\frac{2\sqrt{5}iz(z^6+\sqrt{5}z^3-1)}{(z^3+1)^2} & 0 \ebar \lambda^{-1} \dd z.                  
\edm
\end{example}
The cases $H=10^{-9}$ and $H=1$ are shown in Figure \ref{figurekusner}.

\begin{figure}[ht]
\centering
$
\begin{array}{ccc}\includegraphics[height=27mm]{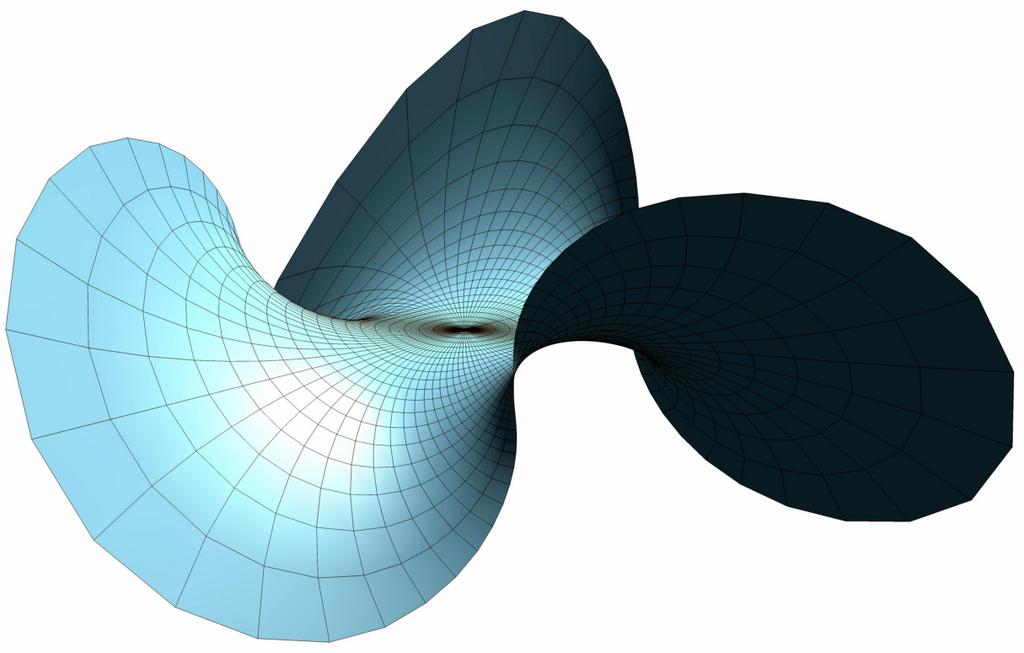} \quad & 
\includegraphics[height=27mm]{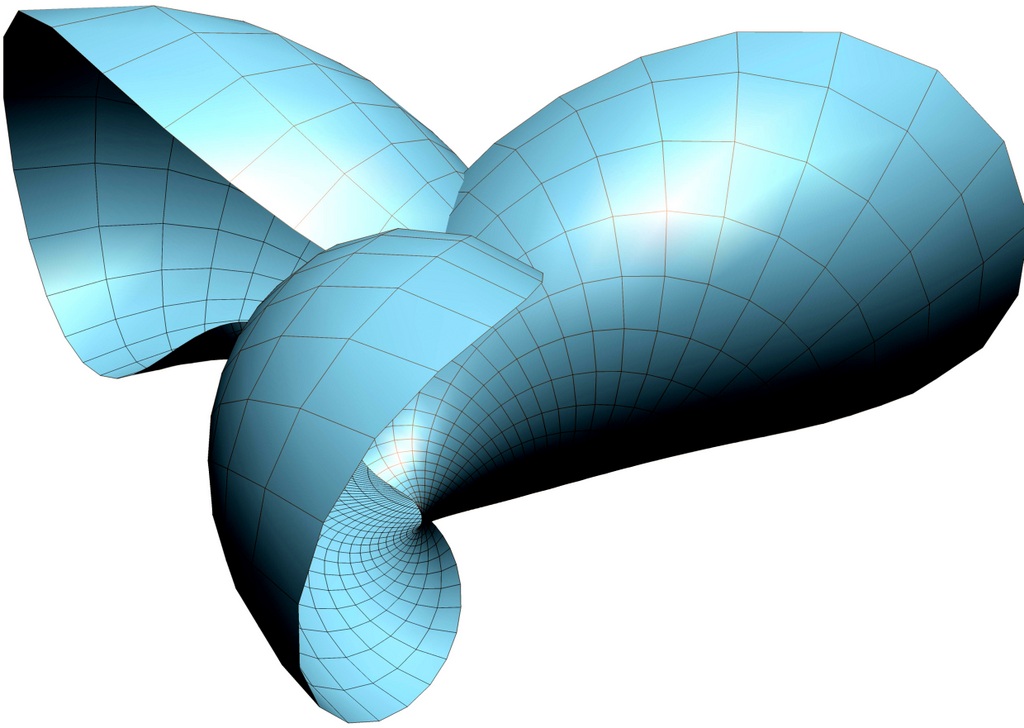} \quad & 
\includegraphics[height=27mm]{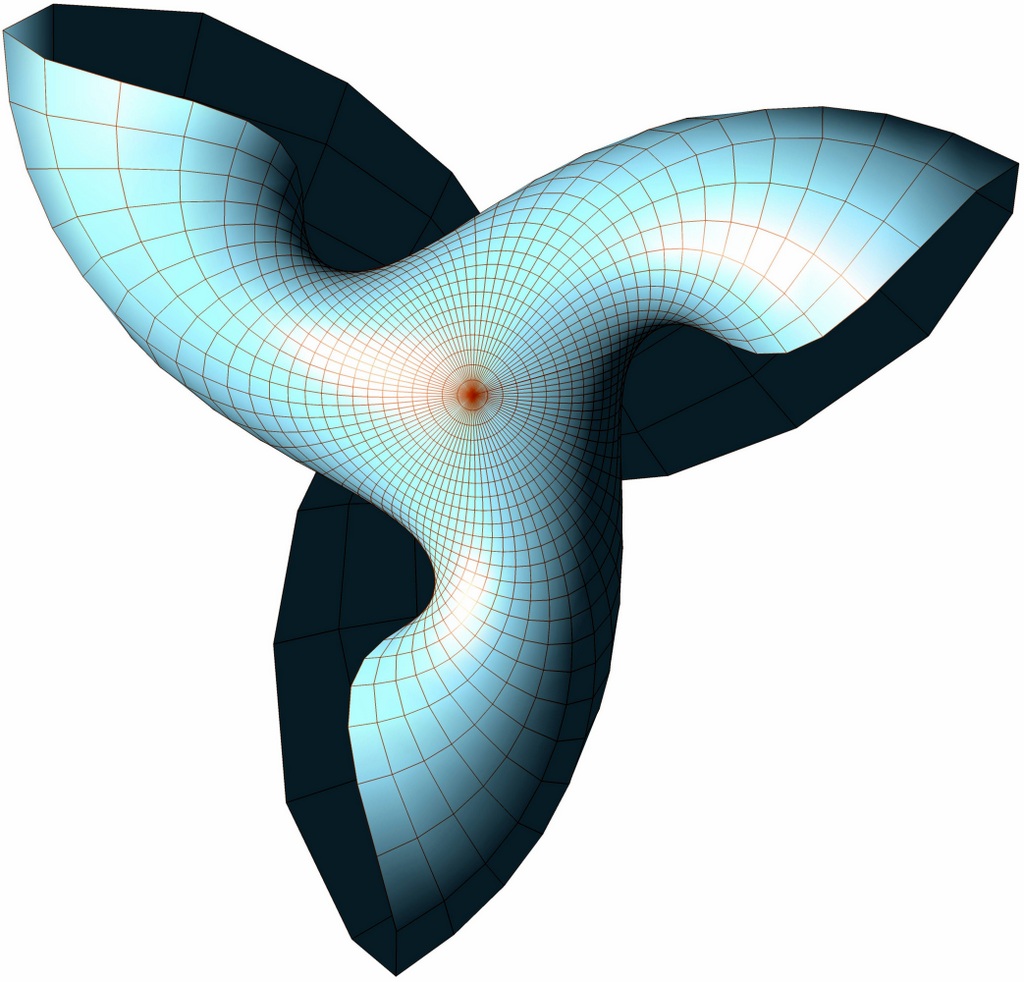} 
 
 \vspace{2ex} \\
 H=10^{-9} &
 H=1 & 
 H=1  
\end{array}
$
\caption{Left: A disc around the symmetry point in an almost minimal version of Kusner's surface with $p=3$. 
Right two images: The CMC $1$ version. } 
\label{figurekusner}
\end{figure}

\providecommand{\bysame}{\leavevmode\hbox to3em{\hrulefill}\thinspace}
\providecommand{\MR}{\relax\ifhmode\unskip\space\fi MR }
\providecommand{\MRhref}[2]{%
  \href{http://www.ams.org/mathscinet-getitem?mr=#1}{#2}
}
\providecommand{\href}[2]{#2}

\end{document}